\documentclass[11pt]{amsart}
\usepackage[colorlinks, linkcolor=blue, citecolor=blue, urlcolor=red!80!gray, pagebackref]{hyperref}
\usepackage{amssymb}
\usepackage{amsmath}
\usepackage{mathtools}
\usepackage{mdwtab}
\usepackage[alphabetic,initials]{amsrefs}
\usepackage{subcaption}
\usepackage{stmaryrd}
\usepackage{tikz}
\usetikzlibrary{positioning,fit}
\usepackage[indent]{parskip}
\usepackage{url}
\usepackage{fullpage}
\usepackage{pstricks,pst-node}

\theoremstyle{plain}
\newtheorem{thm}{Theorem}[section]
\newtheorem{prop}[thm]{Proposition}
\newtheorem{lemma}[thm]{Lemma}
\newtheorem{cor}[thm]{Corollary}

\newtheorem*{thm*}{Theorem}

\theoremstyle{definition}
\newtheorem{dfn}[thm]{Definition}
\newtheorem{ex}[thm]{Example}

\theoremstyle{remark}
\newtheorem{rem}[thm]{Remark}
\numberwithin{equation}{section}

\DeclareMathAlphabet\mathsf{OT1}{cmss}{m}{n}

\newcommand{\fk}{{\mathfrak k}}

\newcommand{\fp}{{\mathfrak p}}

\newcommand{\R}{\mathbb R}           
\newcommand{\Z}{\mathbb Z}           

\def\O{\mathcal{O}}

\newcommand{\su}{\mathfrak{su}}

\newcommand{\so}{\mathfrak{so}}

\newcommand{\g}{\mathfrak{g}}
\renewcommand{\k}{\mathfrak{k}}

\newcommand{\p}{\mathfrak{p}}

\renewcommand{\O}{\mathcal{O}}

\newcommand{\gk}{\operatorname{GKdim}}

\newcommand{\AV}{\mathrm{AV}}

\newcommand{\gD}{\Delta}

\begin{document}

\title[Unitarity of highest weight Harish-Chandra modules]{A characterization of unitarity of some highest weight Harish-Chandra modules}

\author{Zhanqiang Bai}
\address{Department of Mathematical Sciences, Soochow University, Suzhou, Jiangsu, China}
\email{\tt  zqbai@suda.edu.cn}

\author{Markus Hunziker}
\address{
Department of Mathematics, Baylor University, Waco, Texas, USA} 
\email{Markus\_Hunziker@baylor.edu}

\begin{abstract}
Let $L(\lambda)$ be a highest weight Harish-Chandra module with highest weight $\lambda$. When the associated variety of $L(\lambda)$ is not maximal, that is, not equal to the nilradical of the corresponding parabolic subalgebra, we prove that the unitarity of $L(\lambda)$ can be determined by a simple condition on the value of $z = (\lambda + \rho, \beta^{\vee})$, where $\rho$ is half the sum of positive roots and $\beta$ is the highest root. In the proof, certain distinguished antichains of positive noncompact roots play a key role.

By using these antichains, we are also able to provide a uniform formula for the Gelfand--Kirillov dimension of all highest weight Harish-Chandra modules, generalizing our previous result for the case of unitary highest weight Harish-Chandra modules.
\end{abstract}



\maketitle

\tableofcontents

	\section{Introduction}

Let $G_{\mathbb{R}}$ be a connected non-compact simple Lie group with finite center, and let $K_{\mathbb{R}}$ be a maximal compact subgroup. 
From the work of Harish-Chandra (see comments in \cite[\S 3.2]{BHXZ}), it follows that infinite-dimensional highest weight Harish-Chandra modules for $G_\mathbb{R}$ exist if and only if ($G_{\mathbb{R}}$, $K_{\mathbb{R}}$) is a Hermitian symmetric pair.
The problem of determining when a highest weight Harish-Chandra module is unitarizable has been extensively studied by various authors (see, for example, the references in \cite{EHW}). The full classification was independently completed in \cite{EHW} and \cite{Ja1}, though the classification itself is rather intricate.
In this paper, we provide a simple and uniform characterization of unitarity for Harish-Chandra modules with a given associated variety, expressed in terms of the highest weight (see Theorem~\ref{conj}).

From now, we assume that $(G_\R,K_\R)$ is a Hermitian symmetric pair. We denote by $K$ the complexification of the compact group $K_{\mathbb{R}}$
and  by $(\mathfrak{g},\mathfrak{k})$ the complexified Lie algebras of $(G_\R,K_\R)$.
 Then we have the usual decompositition $\mathfrak{g}=\mathfrak{p}^{-}\oplus \mathfrak{k}\oplus \mathfrak{p}^{+}$ of $\mathfrak{g}$ as a $K$-representation. Let $\mathfrak{h} \subseteq \mathfrak{k}$ be a Cartan subalgebra. Then $\mathfrak{h}$ is a Cartan subalgebra of $\mathfrak{g}$.
 Let $\Delta$ and $\Delta(\mathfrak{k})$ denote the root systems of $(\mathfrak{g},\mathfrak{h})$ and $(\mathfrak{k},\mathfrak{h})$, respectively. 
Let $\Delta^+$ be the positive system of $\Delta$, and define 
$\Delta^+(\mathfrak{k})=\Delta(\mathfrak{k})\cap \Delta^+$ and $\Delta(\mathfrak{p}^+)=\Delta^+\setminus \Delta^+(\mathfrak{k})$. 
  Let $\beta $ denote the unique maximal noncompact root of $\Delta^{+}$. Now choose $\zeta  \in \mathfrak{h}^{*} $ so that  $\zeta $ is orthogonal to $\Delta(\mathfrak{k})$  and ($\zeta , \beta ^{\vee } $)=1. Let $\lambda  \in \mathfrak{h}^{*}$ be $\Delta^+(\mathfrak{k})$-dominant integral and $F(\lambda)$ be the irreducible $\mathfrak{k}$-module with highest weight $\lambda $. By letting the nilradical act by zero, we may consider $F(\lambda)$ as a module of the parabolic subalgebra $\mathfrak{q}=\mathfrak{k}+\mathfrak{p}^+$. Then we define:
	$$N(\lambda )=U(\mathfrak{g})\otimes_{U(\mathfrak{q})}F(\lambda ).$$
	Let $L(\lambda) $ denote the irreducible quotient of $ N(\lambda )$, which is a highest weight module of $\mathfrak{g}$.
	
 From \cite{EHW}, $L(\lambda)$ is a highest weight Harish-Chandra module if and only if
$\lambda \in \Lambda^+(\mathfrak{k})$, where
	$$\Lambda^+(\mathfrak{k})=\{\lambda\in \mathfrak{h}^*\mid \lambda \text{~is $\Delta^+(\mathfrak{k})$-dominant~ integral}\}.$$
Write $\rho$ for half the sum of positive roots in $\Delta^+$. Then we can write $\lambda=\lambda_{0}+z \zeta $, with $\lambda_{0} \in \mathfrak{h}^{*} $ such that ($\lambda_{0}  + \rho, \beta  $)=0, and $z=(\lambda+\rho,\beta^\vee) \in \mathbb{R}$.

	The associated variety of a  highest weight Harish-Chandra module is known to be the closure of a single $K$-orbit in $\mathfrak{p}^+$
	(see \S \ref{GKAV} for more details).
	Furthermore, the closures of the $K$-orbits in $\mathfrak{p}^+$ form a linear chain of  varieties
	\begin{equation*}\label{chain}
	\{0\}={\overline{\mathcal{O}_0}}\subset \overline{\mathcal{O}_1}\subset \cdots \subset\overline{\mathcal{O}_{r-1}}\subset \overline{\mathcal{O}_{r}}=\mathfrak{p}^+,
	\end{equation*}
	where $r$ is the $\mathbb{R}$-rank of $G_{\mathbb{R}}$, i.e., the dimension of  a Cartan subgroup of the  group $G_{\mathbb{R}}$, which is also equal to the rank of the symmetric space $G_{\mathbb{R}}/K_{\mathbb{R}}$.
	Therefore, if $L(\lambda)$ is a highest weight Harish-Chandra module with highest weight $\lambda$,
	then there is an integer $0\leq k(\lambda)\leq r$ such that  the associated variety of $L(\lambda)$ is $\overline{\mathcal{O}_{k(\lambda)}}$.

	Denote \begin{equation}\label{z-k}
	   z_{k}=(\rho,\beta^{\vee})+u_{k}=(\rho,\beta^{\vee})-kc, 
	\end{equation}
 for $0\leq k\leq r$. Here $c$ is a real number associated with the Hermitian type Lie group $G_{\mathbb{R}}$, see Table~\ref{constants-k}.
 
	In this paper, we will prove the following result.

	\begin{thm}\label{conj}
		Let  $\lambda\in \Lambda^+(\k)$ and suppose $\AV(L(\lambda))=\overline{\O_k}$ with $0\leq k\leq r-1$. Then   
		$$
		\mbox{$L(\lambda)$ is unitarizable if and only if $(\lambda+\rho,\beta^\vee)= z_k$.}
		$$
	\end{thm}

\begin{rem} 
Note that when $\AV(L(\lambda))=\overline{\O_r}$, it may happen that there is more than one point such that $L(\lambda)$ is unitarizable.
\end{rem}

	In \cite{BH}, we have found a uniform formula for the Gelfand--Kirillov dimensions of unitary highest weight Harish-Chandra modules. Now we want to generalize our formula to all highest weight Harish-Chandra modules.

\begin{dfn}
		For $1\leq k\leq r-1$,  define:
		$$A_k=\{\alpha \in \Delta(\mathfrak{p}^+)\mid \operatorname{ht}(\alpha)=k\left\lceil c \right\rceil+1\}.$$	
	\end{dfn}
Here $\left\lceil c \right\rceil$ denotes the smallest
integer $n$ such that $n \geq c$.
	
\begin{dfn} Let $\lambda_0\in \Lambda^+(\k)$.
	\begin{itemize}
		\item[(a)] If  $\mathfrak{g}$ is of type $ADE$,  for $0\leq k\leq r-1$, define
		$$
		z_k(\lambda_0) :=\min\{z\in \mathbb{Z} \mid \mbox{$\exists ~\alpha \in A_k$ such that $(\lambda_0+z\zeta+\rho,\alpha^\vee)>0$} \}.
		$$
		
		\item[(b)]
		If  $\mathfrak{g}$ is of type $BC$,  for $0\leq k\leq r-1$, define
		$$
		z_k(\lambda_0):=
		\begin{cases}
		\min \{ z\in \mathbb{Z} \mid \mbox{$\exists~ \alpha \in A_k$ such that $(\lambda_0+z\zeta+\rho,\alpha^\vee)>0$} \}, &\mbox{if $k$ is even}\\
		\min \{z\in {\frac{1}{2}}+\mathbb{Z} \mid \mbox{$\exists~ \alpha \in A_k$ such that $(\lambda_0+z\zeta+\rho,\alpha^\vee)>0$} \}, &\mbox{if $k$ is odd.}\\
		\end{cases}
		$$
	\end{itemize}
\end{dfn}

In the special case when $\lambda_0=-(\rho, \beta^{\vee})\zeta$, we also write $z_k$ instead of $z_k(\lambda_0)$
(This coincides with our definition of $z_k$ in (\ref{z-k})).

Our new formula  for the Gelfand--Kirillov dimensions of all  highest weight Harish-Chandra modules is as follows. 

\begin{thm*}[Theorem \ref{main}]
	Suppose $\lambda=\lambda_0+z\zeta$ is a reduction point. Then
	\begin{itemize}

		\item[(a)]  If  $\mathfrak{g}$ is of type $ADE$, then $z\in \mathbb{Z}$ and
		$$
		\gk L(\lambda)=
		\begin{cases}
		rz_{r-1}, & \mbox{if $z<z_{r-1}(\lambda_0)$}\\
		kz_{k-1}, & \mbox{if $z_{k}(\lambda_0)\leq z < z_{k-1}(\lambda_0)$, where $1\leq k\leq r-1$}  \\
		0, &  \mbox{if $z_{0}(\lambda_0) \leq z\in \mathbb{Z}$}.\\
		\end{cases}
		$$
		\item[(b)]
		If  $\mathfrak{g}$ is of type $BC$, then $z\in \mathbb{Z}$ or $z\in \frac{1}{2}+\mathbb{Z}$ and
		\begin{align*}
		\gk L(\lambda)=
		\begin{cases}
		rz_{r-1}, & \mbox{if $z<z_{r-1}(\lambda_0)$}\\
		kz_{k-1}, & \mbox{if $z_{k}(\lambda_0)\leq z < z_{k-2}(\lambda_0)$, where $2\leq k\leq r-1$ and}  \\
		& \mbox{\quad either $z\in \mathbb{Z}$ and  $k$ is even  or $z\in \frac{1}{2}+\mathbb{Z}$ and $k$ is odd}\\
		z_0, & \mbox{if $z_{1}(\lambda_0)  \leq z\in \frac{1}{2}+\mathbb{Z}$}\\
		0, &  \mbox{if $z_{0}(\lambda_0)  \leq z\in \mathbb{Z}$}.\\
		\end{cases}
		\end{align*}
		
	\end{itemize}

\end{thm*}

	\section{Distinguished antichains of positive noncompact roots}
	
	Note that for any integer $1\leq h\leq \operatorname{ht}(\beta)$, the set
	$$
	\Delta(\p^+)_h:=\{\alpha \in \Delta(\mathfrak{p}^+)\mid \operatorname{ht}(\alpha)=h\}
	$$
	is an antichain in $\Delta(\mathfrak{p}^+)$.
	
Let $\Pi$ denote the set of simple roots in $\Delta^+$. 

\begin{lemma}[{\cite[Lemma 4.1]{Ja1}}]
Let $\alpha \in \Delta(\mathfrak{p^{+}})$, let  $\pi_{1},\ldots,\pi_{k}$
be distinct elements of $\Pi\cap \Delta(\mathfrak{k})$, and assume that $\alpha+\pi_{i}\in \Delta(\mathfrak{p^{+}})$ for $i=1,\ldots, k$. Then $k\leq 2$. Furthermore, if $k=2$, then 
$\pi_{1}\perp\pi_{2}$ and  $\alpha+\pi_{1}+\pi_{2}\in \Delta(\mathfrak{p^{+}})$. \hfill$\Box$
\end{lemma}

In light of this lemma, the Hasse diagram of $\Delta(\mathfrak{p^{+}})$ is an upward planar graph 
of order dimension two and hence can be drawn on a two-dimensional orthogonal lattice that has been  
rotated by a $45$-degree angle. 

\begin{ex}
Let
$\mathfrak{g}_{\mathbb{R}}=\su(3,2)$. 
Then we have 
\begin{center}
\begin{pspicture}(-7,-1.5)(3,3)
\uput[l](-2,1.3){$
\mathfrak{p}^+=\{\left(\begin{array}{ccc|cc} 
\cdot&\cdot&\cdot&\, * & * \\[-4 pt]
\cdot&\cdot&\cdot&\, * & * \\[-4 pt]
\cdot&\cdot&\cdot&\, * & * \\[-4 pt]
\hline
\cdot&\cdot&\, \cdot & \cdot & \cdot\\[-4 pt]
\cdot&\cdot&\, \cdot & \cdot & \cdot
\end{array}
\right)\}  
$}
\uput[l](.3,1.3){$\Delta(\mathfrak{p}^+)=$}
\cnode*(1.5,.5){.07}{a0}
\cnode*(1,1){.07}{a1}
\cnode*(2,1){.07}{a2}
\cnode*(.5,1.5){.07}{a3}
\cnode*(1.5,1.5){.07}{a4}
\cnode*(1,2){.07}{a7}
\ncline{->}{a0}{a1}
\ncline{->}{a0}{a2}
\ncline{->}{a1}{a3}
\ncline{->}{a1}{a4}
\ncline{->}{a2}{a4}
\ncline{->}{a2}{a5}
\ncline{->}{a3}{a6}
\ncline{->}{a3}{a7}
\ncline{->}{a4}{a7}
\ncline{->}{a4}{a8}
\ncline{->}{a5}{a8}
\ncline{->}{a6}{a9}
\ncline{->}{a7}{a9}
\ncline{->}{a7}{a10}
\ncline{->}{a8}{a10}
\ncline{->}{a9}{a11}
\ncline{->}{a10}{a11}
\uput[r](1.15,.2){$\alpha_{3}=e_3-e_4$}
\uput[r](.75,2.35){$\beta =e_1-e_5$}
\uput[d](1.1,.9){\scriptsize{2}}
\uput[d](0.6,1.4){\scriptsize{1}}
\uput[u](0.6,1.6){\scriptsize{4}}
\end{pspicture}
\end{center}

\end{ex}

The antichains in $\Delta(\p^+)$ are given in Appendix.

A subset $Y\subseteq \Delta(\fp^+)$ is called a \emph{lower-order ideal} if, for  $\alpha \in \Delta(\fp^+)$ and $\beta \in Y$, $\alpha \leq \beta$ implies that $\alpha \in Y$. 


\begin{dfn}
For $\lambda \in \Lambda^+(\fk)$,
define the \emph{diagram} of $\lambda$ as the set
\begin{equation}
Y_{\lambda} :=\{\alpha\in \Delta(\fp^+)\mid ({\lambda +\rho},{\alpha^\vee})\in \Z_{\leq 0}\},
\end{equation}
viewed as a subposet of $\Delta_\lambda(\fp^+):=\Delta_\lambda \cap \Delta(\fp^+)$, where 
$\Delta_\lambda:=\{\alpha\in \Delta \mid ({\lambda+\rho},{\alpha^\vee})\in \Z\}$
is the integral root system associated to $\lambda$.
\end{dfn}

By \cite[Lem 2.2]{BHXZ},  the poset $Y_{\lambda}$ is a lower order ideal of $\Delta(\fp^+)$ when $\lambda$ is integral.

An \emph{antichain} in a poset is a subset consisting of pairwise noncomparable elements.
The \emph{width} of a poset is the cardinality of maximal antichain in the poset.
 Suppose $\lambda\in \Lambda^+(\mathfrak{k})$, we use
	$m=m(\lambda)$ to denote the width of  $Y_{\lambda}$.



By inspection of the Hasse diagram of $\Delta(\p^+)$, we have the following lemma.

\begin{lemma}\label{antichain}
Suppose $Y\subseteq \Delta(\p^+)$ be a lower order ideal. If $m$ is the width of $Y$, then there exists 
an antichain $A\subseteq Y$ of length $m$ such that all the roots in $A$ 
have the same height.
\end{lemma}

	\begin{lemma}\label{+1}
		Let $\lambda \in \Lambda^+(\k)$. If  $\AV(L(\lambda))=\overline{\O_k}$ with $0\leq k\leq r-1$,  then $(\lambda +\rho,\alpha^\vee)>0$ for some $\alpha \in A_k$.
	\end{lemma}
Recall that $A_k=\Delta(\p^+)_{k\left\lceil c \right\rceil+1}$ is an antichain in $\Delta(\p^+)$.	
	\begin{proof} Suppose $\AV(L(\lambda))=\overline{\O_k}$ with $0\leq k\leq r-1$. 
		
		When $\Delta$ is simply-laced, by our main theorem in \cite{BHXZ}, $\lambda$ is integral and $m(\lambda)=k$.
		Assume that $(\lambda +\rho,\alpha^\vee)\leq 0$ for all $\alpha \in A_k$. Since $A_k$ is an antichain of length $k+1$, it would follow that $m(\lambda)\geq k+1$. Contradiction.
		
		When $\Delta$ is non-simply-laced and $k=2l$ is even, by our main theorem in \cite{BHXZ}, $\lambda$ is integral and $2m(\lambda)=k$. Assume that $(\lambda +\rho,\alpha^\vee)\leq 0$ for all $\alpha \in A_k$. Since $A_k$ is an antichain of length $l+1$, it would follow that $m(\lambda)\geq l+1$. Then $2m(\lambda)\geq 2l+2=k+2>k$. Contradiction!
		
		When $\Delta$ is non-simply-laced and $k=2l+1$ is odd, by our main theorem in \cite{BHXZ}, $\lambda$ is half-integral and $2m(\lambda)+1=k$. Assume that $(\lambda +\rho,\alpha^\vee)\leq 0$ for all $\alpha \in A_k$. Since $A_k$ is an antichain of length $l+1$, it would follow that $m(\lambda)\geq l+1$. Then $2m(\lambda)+1\geq 2l+3=k+2>k$. Contradiction!
		
	\end{proof}

	\section{Gelfand--Kirillov dimension and associated variety}\label{GKAV}

	%
	%
	
	In this section, we will recall some preliminaries on Gelfand--Kirillov dimensions and associated varieties of highest weight modules. See \cites{Vo78, Vo91}  for more details.

	
	Let $M$ be a finite generated $U(\mathfrak{g})$-module. Fix a finite dimensional generating space $M_0$ of $M$. Let $U_{n}(\mathfrak{g})$ be the standard filtration of $U(\mathfrak{g})$. Set $M_n=U_n(\mathfrak{g})\cdot M_0$ and
	\(
	\text{gr} (M)=\bigoplus\limits_{n=0}^{\infty} \text{gr}_n M,
	\)
	where $\text{gr}_n M=M_n/{M_{n-1}}$. Thus $\text{gr}(M)$ is a graded module of $\text{gr}(U(\mathfrak{g}))\simeq S(\mathfrak{g})$.

	
The \textit{Gelfand--Kirillov dimension} of $M$  is defined by
		\begin{equation*}
		\operatorname{GKdim} M = \overline{\lim\limits_{n\rightarrow \infty}}\frac{\log\dim( U_n(\mathfrak{g})M_{0} )}{\log n}.
		\end{equation*}	
The  \textit{associated variety} of $M$ is defined by
		\begin{equation*}
		\AV(M):=\{X\in \mathfrak{g}^* \mid f(X)=0 \text{ for all~} f\in \operatorname{Ann}_{S(\mathfrak{g})}(\operatorname{gr} M)\}.
		\end{equation*}	
These  two definitions are independent of the choice of $M_0$, and $\dim V(M)=\gk M$ (e.g., \cite{NOT}). 
	If $M_0$ is $\mathfrak{a}$-invariant for a subalgebra $\mathfrak{a}\subset\mathfrak{g}$, then 
	\begin{equation}\label{embed}
	\AV(M)\subset (\mathfrak{g}/\mathfrak{a})^*.
	\end{equation}
	
	
	%

	When $M=L(\lambda)$ is a highest weight Harish-Chandra module, we can choose $M_0$ to be the finite dimensional $U(\mathfrak{k})$-module generated by $\mathbb{C}_{\lambda}$. Then $M_0$ is $\mathfrak{k}\oplus\mathfrak{p}^+$-invariant. In view of \eqref{embed},
	\[
	\AV(L(\lambda))\subset(\mathfrak{g}/(\mathfrak{k}\oplus\mathfrak{p}^+))^*\simeq(\mathfrak{p}^-)^*\simeq \mathfrak{p}^+,
	\]
	where the last isomorphism is induced from the Killing form. As shown in \cite{Vo91}, the associated variety $\AV(M)$ is also $K$-invariant. In fact, Yamashita \cite{Hir01} proved that $ \AV(M) $ must be one of $ \overline{\mathcal{O}_{k}} $.
	
	\begin{lemma}
		Let $L(\lambda)$ be a highest weight Harish-Chandra module. Then 
		\begin{equation*}
		\AV(L(\lambda))=\overline{\mathcal{O}_{k(\lambda)}}
		\end{equation*}
		for some $0\leq k(\lambda)\leq r$.
	\end{lemma}

We have the following table from  \cite{EHW}: 
	\begin{center}
		\begin{table}[htb]
			{
				\begin{tabular}
    {cccc}
					\hline
					$\mathfrak{g}_{\R}$ &   $r$ & $c$ & $(\rho,\beta^\vee)$ \\  
					\hline  
					$\mathfrak{su}(p,n-p)$ & $\min\{p, n-p\}$ &  $1$ & $n-1$  \\ 
					$\mathfrak{sp}(n,\mathbb{R})$  & $n$ &   $1/2$ & $n$   \\ 
					$\mathfrak{so}^{*}(2n)$  & $[n/2]$  & $2$ & $2n-3$ \\ 
					$\mathfrak{so}(2,2n-1)$  & $2$  &  $n-3/2$ & $2n-2$ \\ 
					$\mathfrak{so}(2,2n-2)$  & $2$ &  $n-2$& $2n-3$ \\ 
					$\mathfrak{e}_{6(-14)}$  & $2$ &  $3$& $11$ \\ 
					$\mathfrak{e}_{7(-25)}$  & $3$ &  $4$& $17$ \\ 
					\hline
				\end{tabular}
			}
			\caption{Some constants of  Lie groups of Hermitian type}\label{constants-k}
		\end{table}
	\end{center}

	In \cite{BH}, we have found a uniform expression for the GK dimensions and associated varieties of unitary highest weight Harish-Chandra modules.
	
	\begin{prop}[{\cite{BH}}]\label{C: dimYk}
		Suppose $L(\lambda)$ is a unitary highest weight Harish-Chandra module with highest weight $\lambda$. We denote $z=z(\lambda)=(\lambda+\rho,\beta^{\vee})$, then
		\begin{align*}\gk L(\lambda)=
		\begin{cases}
		rz_{r-1}, & \mbox{if  $z<z_{r-1}$}\\
		kz_{k-1}, & \mbox{if  $z=z_{k}=(\rho,\beta^{\vee})-kc, 1\leq k\leq r-1$}\\
		0, & \mbox{if $z=z_{0}=(\rho,\beta^{\vee})$}.
		\end{cases}
		\end{align*}
		
		Denote $k=k(\lambda):=-\frac{(\lambda,  \beta^{\vee} )}{c}$.   Then
		\begin{enumerate}
			\item  If $k>r-1$, we have $\gk L(\lambda)=rz_{r-1}=\frac{1}{2}\dim(G/K).$
			\item If $0\leq k\leq r-1$, then $k$ is a non-negative integer and $$\gk L(\lambda)=k((\rho, \beta^{\vee})-(k-1)c)=kz_{k-1}=\dim \overline{\mathcal{O}_{k(\lambda)}}.$$
		\end{enumerate}
		The associated variety of $L(\lambda)$ is $\overline{\mathcal{O}_{k(\lambda)}}$.

	\end{prop}

	\section{Proof of the main theorem:  simply-laced cases}
	In this section, we assume that $\Delta$ is simply-laced.

	
	
	For  any integer $0\leq k\leq r-1$,  $A_k$ is the antichain in $\Delta(\mathfrak{p}^+)$ such that  $|A_{k}|=k+1$ and $A_k=\Delta(\p^+)_h$ with $h$ minimal.
	
	To prove our Theorem \ref{conj}, we need the following useful lemma.

	\begin{lemma}\label{lamrho-} 
 Let $\tau\in \Lambda^+(\k)$.  If $\lambda_0 = \tau-(\tau+\rho,\beta^\vee)\zeta$ and $\lambda = \lambda_0+z\zeta$, then, for any $\alpha\in \Delta(\p^+)_h$,  
		$$
		(\lambda+\rho,\alpha^\vee) =  z-(\rho,\beta^\vee) +h - (\tau,\beta^\vee -\alpha^\vee).
		$$
	\end{lemma}
	\begin{proof}
		Let $\alpha\in \Delta(\p^+)_h$. Since $\Delta$ is simply-laced,  $(\zeta,\alpha^\vee)=1$ and $(\rho,\alpha^\vee)=h$. Thus,
		\begin{align*}
		(\lambda+\rho,\alpha^\vee) &= (\lambda_0+z\zeta+\rho,\alpha^\vee)\\ &=  (\tau,\alpha^\vee) -(\tau+\rho,\beta^\vee)(\zeta,\alpha^\vee) +z(\zeta,\alpha^\vee)+(\rho,\alpha^\vee)\\
		& = (\tau,\alpha^\vee) -(\tau+\rho,\beta^\vee)  +z +h \\
		& = (\tau,\alpha^\vee-\beta^\vee) -(\rho,\beta^\vee)  +z +h \\
		&= z-(\rho,\beta^\vee) +h - (\tau,\beta^\vee -\alpha^\vee).
		\end{align*}
	\end{proof}

Now we can prove our Theorem \ref{conj}. The idea is very simple. From Lemma \ref{+1}, since $\AV(L(\lambda))=\overline{\mathcal{O}_{k}}$, we have $(\lambda +\rho,\alpha^\vee)>0$ for some $\alpha \in A_k= \Delta(\mathfrak{p^{+}})_{k\left\lceil c \right\rceil+1}$. 
Then $\operatorname{ht}(\alpha)=kc+1$. So if $(\lambda +\rho,\beta^\vee)=z=z_k=(\rho,\beta^\vee)-kc$, we will have 
\begin{align*}
		(\lambda+\rho,\alpha^\vee) &= z-(\rho,\beta^\vee) +h - (\tau,\beta^\vee -\alpha^\vee)\\
  &=(\rho,\beta^\vee)-kc-(\rho,\beta^\vee) +kc+1 - (\tau,\beta^\vee -\alpha^\vee)\\
  &=1-(\tau,\beta^\vee -\alpha^\vee)>0.
		\end{align*}
This condition is very restrictive. In the following, we will give a case-by-case discussion for this condition, which will imply the unitarity of $L(\lambda)$.

	\subsection{Case $\g_\R=\su(p,q)$}
	For  $0\leq k\leq r-1=\min\{p,q\}-1$,
	$$
	A_k=\Delta(\p^+)_{k+1}=\{\varepsilon_{p-k+i} -\varepsilon_{p+i+1} = [\, \underbrace{0,\ldots,0}_{p-k+i-1},\underbrace{1,\ldots,1}_{k+1} ,\underbrace{0,\ldots,0}_{q-i-1}\, ] \mid 0\leq i \leq k\}.
	$$
	(Here and in the following we write $[n_1,n_2,\ldots]:=n_1\alpha_1+n_2\alpha_2+\cdots$, where $\alpha_1,\alpha_2,\ldots$ are the simple roots.)
	
	\medskip\noindent
	Fix $\tau =a_1\omega_1+\cdots + a_{p-1}\omega_{p-1}+b_{q-1}\omega_{p+1}+\cdots + b_1\omega_{p+q-1}\in \Lambda^+(\k)$. 
	Then 
	$$
	\lambda_0  = \tau-(p+q-
	1+a_1+\cdots a_{p-1}+ b_{q-1}+\cdots +b_1)\zeta.
	$$
	For $\alpha=\varepsilon_{p-k+i} -\varepsilon_{p+i+1}$,   
	$$
	(\lambda_0 +z\zeta +\rho,\alpha^\vee)= (z-(p+q-k-2))- (a_1+\cdots+a_{p-k+i-1}+b_{q-i-1}+\cdots +b_1).
	$$
	Now suppose $\lambda =\lambda_0+z\zeta$ such that  $\AV(L(\lambda))=\overline{\mathcal{O}_{k}}$ with $0\leq k\leq r-1$ and $z= z_k=p+q-k-1$.
	By Lemma \ref{+1},  $(\lambda +\rho,\alpha^\vee)>0$ for some $\alpha \in A_k$. For $\alpha=\varepsilon_{p-k+i} -\varepsilon_{p+i+1}$, by Lemma \ref{lamrho-}   we have 
	$$
	(\lambda +\rho,\alpha^\vee) =1- (a_1+\cdots+a_{p-k+i-1}+b_{q-i-1}+\cdots +b_1) 
	$$
	and hence $(\lambda +\rho,\alpha^\vee)>0$ implies  $a_1=\cdots= a_{p-k+i-1}=0$ and $b_1=\cdots = b_{q-i-1}=0$. It follows that $\lambda$ is of the form
	$$
	\lambda=(\underbrace{-k,\dots,-k}_{\geq 1},\underbrace{*,\ldots,* \,; *,\ldots,*}_k,\underbrace{0,\ldots,0}_{\geq 1}).
	$$
	By \cite{KV}, $L(\lambda)$ is unitary.
	%

	\subsection{Case $\g_\R=\so^*(2n)$}  For  $1\leq k\leq r-1=[\frac{n}{2}]-1$,
	\begin{align*}
	A_k&=\Delta(\mathfrak{p}^+)_{2k+1}=\{[\, \underbrace{0,\ldots,0}_{n-2k-2},\underbrace{1,\ldots,1}_{2k} ,0,1\, ],\\
	&~~~~~~~~~\quad\quad\quad\quad\quad\quad\quad[\, \underbrace{0,\ldots,0}_{n-2k-1+i},\underbrace{1,\ldots,1}_{2k-2i-1},\underbrace{2,\ldots,2}_{i} , 1,1\, ] \mid 0\leq i \leq k-1\},\\
	A_0&= \Delta(\mathfrak{p}^+)_{1}=\{[0,0,\dots,0,1]\}.
	\end{align*}
	Fix $\tau=a_1\omega_1+a_2\omega_2+\cdots +a_{n-1}\omega_{n-1}\in \Lambda^+(\k)$.  Then, since  $\beta=[1, 2, \cdots, 2, 1,1]$,
	$$
	\lambda_0  =\tau-(\tau+\rho,\beta^{\vee})\zeta= \tau-(2n-3+a_1+2a_2+\cdots+2a_{n-2}+ a_{n-1})\zeta.
	$$

	Suppose $\AV(L(\lambda))=\overline{\mathcal{O}_{0}}$ and $z(\lambda)=z_0=2n-3$. By  Lemma \ref{+1} and Lemma \ref{lamrho-},  
	$$1 -(a_1+2a_2+\cdots+2a_{n-2}+a_{n-1})>0$$ 
	and hence $a_1=a_2=\cdots=a_{n-1}=0$. It follows that $\lambda=0$ (trivial representation).

	Now suppose $\lambda =\lambda_0+z\zeta$ such that  $\AV(L(\lambda))=\overline{\mathcal{O}_{k}}$ with $1\leq k\leq r-1$ and $z= z_k=2n-3-2k$.
	By Lemma \ref{+1},  $(\lambda +\rho,\alpha^\vee)>0$ for some $\alpha \in A_k$. 
	So if $n-2k>2$, by Lemma \ref{lamrho-} we have $$1-(a_1+2a_2+\cdots+2a_{n-2k-2}
	+a_{n-2k-1}+\cdots+a_{n-2}+a_{n-1})>0$$
	or $$1-(a_1+2a_2+\cdots+2a_{n-2k-1+i}
	+a_{n-2k+i}+\cdots+a_{n-i-2})>0.$$
	
	In the first case, $a_1=a_2=\cdots=a_{n-1}=0$,  and hence $\lambda=-2k\zeta$ ($k$-th Wallach representation).

	In the second case, $a_1=a_2=\cdots=a_{n-i-2}=0$. 
	It follows that \begin{align*}
	\lambda=&a_{n-i-1}\omega_{n-i-1}+\cdots +a_{n-1}\omega_{n-1}\\
	&-(2k+2a_{n-i-1}+\cdots +2a_{n-2}+a_{n-1})\zeta\\
	=&a_{n-i-1}\omega_{n-i-1}+\cdots +a_{n-1}\omega_{n-1}\\
	&-(2k+2a_{n-i-1}+\cdots +2a_{n-2}+a_{n-1})\omega_{n}.
	\end{align*}
	
	So
	$\lambda$ is of the form
	$$
	\lambda=(\underbrace{-k,\dots,-k}_{n-i- 1},\underbrace{*,\ldots,* }_{i+1}).
	$$
	By \cite{EW}, \cite{EHW} or \cite{DES}, $L(\lambda)$ is unitary.
	
	So if $n-2k=2$, by Lemma \ref{lamrho-} we have $$1-(a_2+\cdots+a_{n-1})>0$$
	or $$1-(a_1+2a_2+\cdots+2a_{n-2k-1+i}
	+a_{n-2k+i}+\cdots+a_{n-i-2})>0.$$
	
	In the first case, $a_2=\cdots=a_{n-1}=0$,  and hence $\lambda=a_1\omega_{1}+(2-n-a_1)\zeta$ (unitary reduction point).

	In the second case, $a_1=a_2=\cdots=a_{n-i-2}=0$. 
	It follows that \begin{align*}
	\lambda=&a_{n-i-1}\omega_{n-i-1}+\cdots +a_{n-1}\omega_{n-1}\\
	&-(2k+2a_{n-i-1}+\cdots +2a_{n-2}+a_{n-1})\zeta\\
	=&a_{n-i-1}\omega_{n-i-1}+\cdots +a_{n-1}\omega_{n-1}\\
	&-(2k+2a_{n-i-1}+\cdots +2a_{n-2}+a_{n-1})\omega_{n}.
	\end{align*}
	
	So
	$\lambda$ is of the form
	$$
	\lambda=(\underbrace{-k,\dots,-k}_{n-i- 1},\underbrace{*,\ldots,* }_{i+1}).
	$$
	By \cite{EW}, \cite{EHW} or \cite{DES}, $L(\lambda)$ is unitary.

	\subsection{Case $\g_\R=\so(2,2n-2)$}  For  $0\leq k\leq r-1=2-1=1$,
	\begin{align*}
	A_0&=\Delta(\mathfrak{p}^+)_{1}=\{\varepsilon_{1}-\varepsilon_{2}=[1,0,...,0]\},\\
	A_1&= \Delta(\mathfrak{p}^+)_{n-1}=\{\varepsilon_{1}-\varepsilon_{n}= [1,\dots,1,0], \varepsilon_{1}+\varepsilon_{n}= [1,\dots,1,0,1]\}.
	\end{align*}
	
	Fix $\tau=a_2\omega_2+\cdots +a_{n-1}\omega_{n-1}+a_n\omega_n\in \Lambda^+(\k)$.  Then, since  $\beta=\varepsilon_{1}+\varepsilon_{2}=[1,2, \cdots, 2,1,1]$,
	$$
	\lambda_0  =\tau-(\tau+\rho,\beta^{\vee})\zeta= \tau-(2n-3+2a_2+\cdots+2a_{n-2}+a_{n-1}+a_n)\zeta.
	$$

	%
	%
	%
	%
	%
	
	
	Suppose $\AV(L(\lambda))=\overline{\mathcal{O}_{0}}$ and $z(\lambda)=z_0=2n-3$. By  Lemma \ref{+1} and Lemma \ref{lamrho-},  
	$$1 -(2a_2+\cdots+2a_{n-2}+a_{n-1}+a_{n})>0$$ 
	and hence $a_2=\cdots=a_{n}=0$. It follows that $\lambda=0$ (trivial representation). 
	
	Suppose $\AV(L(\lambda))=\overline{\mathcal{O}_{1}}$ and $z(\lambda)=z_1=n-1$. By  Lemma \ref{+1} and Lemma \ref{lamrho-},  
	$$1 -(a_2+\cdots+a_{n-1})>0$$ 
	or $$1 -(a_2+\cdots+a_{n-2}+a_{n})>0.$$

	In the first case, we have $a_2=\cdots=a_{n-1}=0$. Hence
	$\lambda=a_n\omega_{n}-(n-2+a_n)\zeta$ (unitary reduction point). 
	
	In the second case, we have  $a_2=\cdots=a_{n-2}=a_n=0$. Hence
	$\lambda=a_{n-1}\omega_{n-1}-(n-2+a_{n-1})\zeta$ (unitary reduction point).

	\subsection{Case $\g_\R=\mathfrak{e}_{6(-14)}$} By inspection of the Hasse diagram of $\Delta(\p^+)$,
	\begin{align*}
	A_0&= \Delta(\mathfrak{p}^+)_{1}=\{[1,0,0,0,0,0]\},\\ 
	A_1&= \Delta(\mathfrak{p}^+)_{4}=\{ [1, 0, 1, 1, 1, 0],[1, 1, 1, 1, 0, 0]\}.
	\end{align*}
	Fix $\tau=a_2\omega_2+a_3\omega_3+a_4\omega_4+a_5\omega_5+a_6\omega_6$. Then, since  $\beta=[1, 2, 2, 3, 2, 1]$,
	$$
	\lambda_0 =\tau-(\tau+\rho,\beta^\vee) \zeta =\tau-(11+ 2a_2+ 2a_3+ 3a_4+ 2a_5+a_6)\zeta.
	$$
	
	Suppose $\AV(L(\lambda))=\overline{\mathcal{O}_{0}}$ and $z(\lambda)=z_0=11$. By  Lemma \ref{+1} and Lemma \ref{lamrho-},  
	$$1-(2a_2+2a_3+3a_4+2a_5+a_6)>0$$ and hence $a_2=a_3=a_4=a_5=a_6=0$. It follows that $\lambda=0$ (trivial representation). 
	
	\medskip\noindent
	Suppose $\AV(L(\lambda))=\overline{\mathcal{O}_{1}}$ and $z(\lambda)=z_1=8$. By  Lemma \ref{+1} and Lemma \ref{lamrho-},  
	$$
	1-(2a_2+a_3+2a_4+a_5+a_6)>0 \ \mbox{or}\ 
	1-(a_2+ a_3+ 2a_4+ 2a_5+a_6)>0.
	$$
	In either case, $a_2=a_3=a_4=a_5=a_6=0$ and hence $\lambda=-3\zeta$ (1st Wallach representation).

	\subsection{Case $\g_\R=\mathfrak{e}_{7(-25)}$}
	By inspection of the Hasse diagram of $\Delta(\p^+)$,
	\begin{align*}
	A_0&= \Delta(\mathfrak{p}^+)_{1}=\{[0,0,0,0,0,0,1]\}, \\
	A_1&= \Delta(\mathfrak{p}^+)_{5}=\{[0, 0, 1, 1, 1, 1, 1], [0, 1, 0, 1, 1, 1, 1]\},\\
	A_2&=\Delta(\mathfrak{p}^+)_{9}=\{[1, 1, 2, 2, 1, 1, 1],[1, 1, 1, 2, 2, 1, 1],[0, 1, 1, 2, 2, 2, 1]\}.
	\end{align*}
	Fix $\tau=a_1\omega_2+a_2\omega_2+a_3\omega_3+a_4\omega_4+a_5\omega_5+a_6\omega_6$. Then, since $\beta=[2, 2, 3, 4, 3, 2, 1]$,
	$$
	\lambda_0=\tau-(17+2a_1+2a_2+3a_3+4a_4+3a_5+2a_6)\zeta.
	$$
	
	Suppose $\AV(L(\lambda))=\overline{\mathcal{O}_{0}}$ and $z(\lambda)=z_0=17$. By Lemma \ref{+1} and Lemma \ref{lamrho-}, 
	$$1-(2a_1+2a_2+3a_3+4a_4+3a_5+2a_6) >0$$ and hence $a_1=a_2=a_3=a_4=a_5=a_6=0$. It follows that $\lambda=0$ (trivial representation). 
	
	\medskip\noindent
	Suppose $\AV(L(\lambda))=\overline{\mathcal{O}_{1}}$ and $z(\lambda)=z_1=13$. By Lemma \ref{+1} and Lemma \ref{lamrho-}, 
	$$
	1-(2a_1+2a_2+ 2a_3+ 3a_4+ 2a_5+ a_6)>0 \ \mbox{or}\ 
	1-(2a_1+ a_2+ 3a_3+ 3a_4+ 2a_5+ a_6)>0.
	$$
	In either case, $a_1=a_2=a_3=a_4=a_5=a_6=0$ and hence $\lambda=-4\omega_7$ (1st Wallach representation). 
	
	\medskip\noindent
	Suppose $\AV(L(\lambda))=\overline{\mathcal{O}_{2}}$ and $z(\lambda)=z_2=9$. By Lemma \ref{+1} and Lemma \ref{lamrho-},  
	\begin{align*}
	&1-(a_1+ a_2 + a_3 + 2a_4 + 2a_5 +a_6>0,\\
	&  1- (a_1+ a_2+ 2a_3+ 2a_4+ a_5+a_6)>0,\\
	\ \text{or~}&1-(2a_1+ a_2+ 2a_3+ 2a_4+ a_5)>0.
	\end{align*}
	In the first two cases, $a_1=a_2=a_3=a_4=a_5=a_6=0$ and hence $\lambda=-8\zeta$ (2nd Wallach representation).
	In the third case, $a_1=a_2=a_3=a_4=a_5=0$ and $\lambda = a_6\omega_6+(-2a_6 -8)\zeta$ (unitary reduction point).

	\section{Proof of the main theorem:  non-simply-laced cases}
	In this section, we assume that $\Delta$ is not simply-laced.
	
	For  any integer $0\leq k\leq r-1$,  $A_k$ is the antichain in $\Delta(\mathfrak{p}^+)$ such that  $|A_{k}|=[\frac{k}{2}]+1$ and $A_k=\Delta(\p^+)_h$ with $h$ minimal.
	
The proof of Theorem \ref{conj} in non-simply-laced cases is similar to the simply-laced cases. We need Lemma \ref{+1} and the following useful lemma in the computation.

	\begin{lemma}\label{lamrho-2} Let $\tau\in \Lambda^+(\k)$.  If $\lambda_0 = \tau-(\tau+\rho,\beta^\vee)\zeta$ and $\lambda = \lambda_0+z\zeta$, then, for any $\alpha\in \Delta(\p^+)_h$,  
		$$
		(\lambda+\rho,\alpha^\vee) =  z-(\rho,\beta^\vee) +(\rho,\alpha^\vee) - (\tau,\beta^\vee -\alpha^\vee)  \text{~when $\alpha$ is a long root},
		$$
		and 	$$
		(\lambda+\rho,\alpha^\vee) =  2z-2(\rho,\beta^\vee) +(\rho,\alpha^\vee) - (\tau,2\beta^\vee -\alpha^\vee) \text{~when $\alpha$ is a short root}.$$
	\end{lemma}
	\begin{proof}
		Let $\alpha\in \Delta(\p^+)_h$. Since $\Delta$ is not simply-laced,  $(\zeta,\alpha^\vee)=1$ if $\alpha$ is a long root and  $(\zeta,\alpha^\vee)=2$ if $\alpha$ is a short root.  $(\rho,\alpha^\vee)=h$. Thus, if  $\alpha$ is a long root, we have
		\begin{align*}
		(\lambda+\rho,\alpha^\vee) &= (\lambda_0+z\zeta+\rho,\alpha^\vee)\\ &=  (\tau,\alpha^\vee) -(\tau+\rho,\beta^\vee)(\zeta,\alpha^\vee) +z(\zeta,\alpha^\vee)+(\rho,\alpha^\vee)\\
		& = (\tau,\alpha^\vee) -(\tau+\rho,\beta^\vee)  +z +(\rho,\alpha^\vee) \\
		& = (\tau,\alpha^\vee-\beta^\vee) -(\rho,\beta^\vee)  +z +(\rho,\alpha^\vee) \\
		&= z-(\rho,\beta^\vee) +(\rho,\alpha^\vee) - (\tau,\beta^\vee -\alpha^\vee).
		\end{align*}
		
		If  $\alpha$ is a short root, we have
		\begin{align*}
		(\lambda+\rho,\alpha^\vee) &= (\lambda_0+z\zeta+\rho,\alpha^\vee)\\ &=  (\tau,\alpha^\vee) -(\tau+\rho,\beta^\vee)(\zeta,\alpha^\vee) +z(\zeta,\alpha^\vee)+(\rho,\alpha^\vee)\\
		& = (\tau,\alpha^\vee) -2(\tau+\rho,\beta^\vee)  +2z +(\rho,\alpha^\vee) \\
		& = (\tau,\alpha^\vee-2\beta^\vee) -2(\rho,\beta^\vee)  +2z +(\rho,\alpha^\vee) \\
		&= 2z-2(\rho,\beta^\vee) +(\rho,\alpha^\vee) - (\tau,2\beta^\vee -\alpha^\vee).
		\end{align*}
	\end{proof}

	\subsection{$\mathfrak{g}_{\R}=\mathfrak{sp}(n,\mathbb{R})$}
	For  $0\leq k\leq r-1=n-1$,
	\begin{align*}
	\text{if $k=2m$ is even}, A_k&=\Delta(\mathfrak{p}^+)_{k+1}=\{2\varepsilon_{n-m},\varepsilon_{n-m-i}+\varepsilon_{n-m+i}\mid 1\leq i\leq m\}\\
	&=\{[\, \underbrace{0,\ldots,0}_{n-m-1-i},\underbrace{1,\ldots,1}_{2i},\underbrace{2,\ldots,2}_{m-i},1\, ]\mid 0\leq i \leq m\},
	\end{align*}
	\begin{align*}
	\text{if $k=2m+1$ is odd}, A_k&=\Delta(\mathfrak{p}^+)_{k+1}=\{\varepsilon_{n-m-1-i}+\varepsilon_{n-m+i}\mid 0\leq i\leq m\}\\
	&=\{[\, \underbrace{0,\ldots,0}_{n-m-2-i},\underbrace{1,\ldots,1}_{1+2i},\underbrace{2,\ldots,2}_{m-i},1\, ]\mid 0\leq i \leq m\}.
	\end{align*}
	Fix $\tau=a_1\omega_1+a_2\omega_2+\cdots +a_{n-1}\omega_{n-1}\in \Lambda^+(\k)$.  Then, since  $\beta=[2, \cdots, 2, 1]$,
	$$
	\lambda_0  =\tau-(\tau+\rho,\beta^{\vee})\zeta= \tau-(n+a_1+a_2+\cdots+a_{n-2}+ a_{n-1})\zeta.
	$$

	Now suppose $\lambda =\lambda_0+z\zeta$ such that  $\AV(L(\lambda))=\overline{\mathcal{O}_{k}}$ with $1\leq k\leq r-1$ and $z= z_k=n-\frac{k}{2}$.
	By Lemma \ref{+1},  $(\lambda +\rho,\alpha^\vee)>0$ for some $\alpha \in A_k$. 
	So if $k=2m$ is even, by Lemma \ref{lamrho-2} we have $$1-(a_1+a_2+\cdots+a_{n-m-1})>0$$
	or $$2-(2a_1+\cdots+2a_{n-m-1-i}+a_{n-m-i}+\cdots+a_{n-1-m+i})>0.$$
	In the first  case, $a_1=a_2=\cdots =a_{n-m-1}=0$ and hence \begin{align*}
	\lambda&=a_{n-m}\omega_{n-m}+\cdots+a_{n-1}\omega_{n-1}-(a_{n-m}+\cdots+a_{n-1})\zeta-\frac{1}{2}\zeta\\
	&=(\underbrace{0,\ldots,0}_{n-m},\underbrace{*,\ldots,*}_{m})-\frac{1}{2}\zeta.
	\end{align*}
	By \cite{EHW} or \cite{EW}, $L(\lambda)$ is unitary.
	
	In the second case, $a_1=a_2=\cdots =a_{n-m-1-i}=0$ and 
	$a_{n-m-i}=\cdots=a_{n-1-m+i}=0$ or at most one of $\{a_{n-m-i},\cdots,a_{n-1-m+i}\}$ equals to $1$ with the rest  queal to $0$. Hence
	\begin{align*}
	\lambda&=a_{n-m-i}\omega_{n-m-i}+\cdots+a_{n-1}\omega_{n-1}-(a_{n-m-i}+\cdots+a_{n-1})\zeta-\frac{1}{2}\zeta\\
	&=(\underbrace{0,\ldots,0}_{\geq n-m-i},\underbrace{-1,\cdots,-1,\overbrace{*,\ldots,*}^{m-i}}_{\leq m+i})-\frac{1}{2}\zeta.
	\end{align*}
	
	By \cite{EHW} or \cite{EW}, $L(\lambda)$ is unitary.

	
	If  $k=2m+1$ is odd, we have
	$$2-(2a_1+\cdots+2a_{n-m-2-i}+a_{n-m-1-i}+\cdots+a_{n-1-m+i})>0,$$
	then the arguments are similar to the above case.
	
	%
	%
	%
	%
	%
	%
	
	\subsection{$\mathfrak{g}_{\R}=\mathfrak{so}(2,2n-1)$}
	For  $0\leq k\leq r-1=2-1=1$,
	\begin{align*}
	A_0&=\Delta(\mathfrak{p}^+)_{1}=\{\varepsilon_{1}-\varepsilon_{2}=[1,0,...,0]\},\\
	A_1&= \Delta(\mathfrak{p}^+)_{n}=\{\varepsilon_{1}= [1,1,\dots,1]\}.
	\end{align*}
	
	Fix $\tau=a_2\omega_2+\cdots +a_{n-1}\omega_{n-1}+a_n\omega_n\in \Lambda^+(\k)$.  Then, since  $\beta=\varepsilon_{1}+\varepsilon_{2}=[1,2, \cdots, 2]$,
	$$
	\lambda_0  =\tau-(\tau+\rho,\beta^{\vee})\zeta= \tau-(2n-2+2a_2+\cdots+2a_{n-1}+a_{n})\zeta.
	$$
	

	%
	%
	%
	%
	%
	
	
	Suppose $\AV(L(\lambda))=\overline{\mathcal{O}_{0}}$ and $z(\lambda)=z_0=2n-2$. By  Lemma \ref{+1} and  Lemma \ref{lamrho-2},  
	$$1 -(2a_2+\cdots+2a_{n-1}+a_{n})>0$$ 
	and hence $a_2=\cdots=a_{n}=0$. It follows that $\lambda=0$ (trivial representation). 
	
	Suppose $\AV(L(\lambda))=\overline{\mathcal{O}_{1}}$ and $z(\lambda)=z_1=n-\frac{1}{2}$. By  Lemma \ref{+1} and  Lemma \ref{lamrho-2},  
	$$2 -(2a_2+\cdots+2a_{n-1}+a_{n})>0$$ 
	and hence $a_2=\cdots=a_{n}=0$ or $a_2=\cdots=a_{n-1}=a_n-1=0$. 
	
	In the first case, it follows that $\lambda=-(n-\frac{3}{2})\zeta$ (1st Wallach representation). 
	
	In the second case, we have 
	$\lambda=\omega_{n}-(n-\frac{1}{2})\zeta$ (unitary reduction point).

	\section{A uniform formula for the Gelfand--Kirillov dimension }
In our previous paper \cite{BH}, we found a uniform formula for the Gelfand--Kirillov dimensions of all unitary highest weight modules. Now we will give a new formula for  the Gelfand--Kirillov dimensions of all  highest weight Harish-Chandra modules. 

%

%
%
%
%
%
%

We recall the definition of $z_k(\lambda_0)$ in the introduction. Then we have the following lemma.

\begin{lemma}
Suppose    $\mathfrak{g}$ is of type $ADE$. For $z=(\lambda+\rho, \beta^{\vee})\in \mathbb{Z}$, we have
 $$z_k(\lambda_0)\leq z \iff  m(\lambda)\leq k.$$
\end{lemma}	
\begin{proof}
	 	It is easy to verify (case-by-case) that
	$$m(\lambda)\geq k+1 \iff A_k \subset \{\alpha \in  \Delta(\mathfrak{p}^+) \mid (\lambda+\rho, \alpha^{\vee}
	) \in \mathbb{Z}_{\leq 0}\}.$$

	 For $z\in \mathbb{Z}$, we have
	\begin{align*}
	z_k(\lambda_0)\leq z &\iff  \exists~ \alpha \in A_{k} \text{~such~that~} (\lambda_0+ z\zeta+\rho, \alpha^{\vee})>0\\
	&\iff A_{k} \nsubseteq \{\alpha \in  
	\Delta(\mathfrak{p}^+)\mid (\lambda+\rho,\alpha^{\vee})\leq 0\}\\
	&\iff m(\lambda)<k+1=|A_k|\\
	&\iff  m(\lambda)\leq k.
	\end{align*}
\end{proof}

We recall the main theorem in \cite{BHXZ}.
\begin{prop}[\cite{BHXZ}]
	Suppose $L(\lambda)$  is a highest weight Harish-Chandra module with highest weight $\lambda$ and $\AV(L(\lambda))=\overline{\mathcal{O}_{k(\lambda)}}$. Let 
$m=\operatorname{width}(Y_\lambda)$. Then $k(\lambda)$ is given as follows.
\begin{enumerate}
 \item[(a)] If $\gD$ is simply laced and $\lambda$ is integral, then $k(\lambda)=m$.
 \item[(b)] If $\gD$ is non-simply laced and $\lambda$ is integral, then
 \begin{equation*}
  k(\lambda)=\begin{cases} 2m, &\text{if }m< \frac{r+1}{2}  \\
                       r, &\text{if }m=\frac{r+1}{2}.
         \end{cases}
 \end{equation*}
\item[(c)] If $\gD$ is non-simply laced and $\lambda$ is half-integral, then
 \begin{equation*}
  k(\lambda)=\begin{cases} 2m+1, &\text{if }m< \frac{r}{2}  \\
                       r, &\text{if }m=\frac{r}{2}.
         \end{cases}
 \end{equation*}
\item[(d)] In all other cases $k(\lambda)=r$.
\end{enumerate}

\end{prop}

In the following, we give the new formula for the Gelfand--Kirillov dimensions of all  highest weight Harish-Chandra modules. 
	
	\begin{thm}\label{main}
		Suppose $\lambda=\lambda_0+z\zeta$ is a reduction point. Then
	\begin{itemize}

		\item[(a)]  If  $\mathfrak{g}$ is of type $ADE$, then $z\in \mathbb{Z}$ and
		$$\gk L(\lambda)=
		\begin{cases}
		rz_{r-1}, & \mbox{if $z<z_{r-1}(\lambda_0)$}\\
		kz_{k-1}, & \mbox{if $z_{k}(\lambda_0)\leq z < z_{k-1}(\lambda_0)$, where $1\leq k\leq r-1$}  \\
		0, &  \mbox{if $z_{0}(\lambda_0) \leq z\in \mathbb{Z}$}.	\end{cases}
		$$
		\item[(b)]
		If  $\mathfrak{g}$ is of type $BC$, then $z\in \mathbb{Z}$ or $z\in \frac{1}{2}+\mathbb{Z}$ and
  		\begin{align*}
		 \gk L(\lambda)=
		\begin{cases}
		rz_{r-1}, & \mbox{if $z<z_{r-1}(\lambda_0)$}\\
		kz_{k-1}, & \mbox{if $z_{k}(\lambda_0)\leq z < z_{k-2}(\lambda_0)$, where $2\leq k\leq r-1$ and}  \\
		& \mbox{\quad either $z\in \mathbb{Z}$ and  $k$ is even  or $z\in \frac{1}{2}+\mathbb{Z}$ and $k$ is odd}\\
		z_0, & \mbox{if $z_{1}(\lambda_0)  \leq z\in \frac{1}{2}+\mathbb{Z}$}\\
		0, &  \mbox{if $z_{0}(\lambda_0)  \leq z\in \mathbb{Z}$}.\\
		\end{cases}
		\end{align*}

		\end{itemize}

\end{thm}
	
	\begin{proof}
		In the following, suppose $0\leq k\leq r-1$.
		
	(a)	In type $ADE$, if $\lambda=\lambda_0+z\zeta$ is a reduction point, then $z\in \mathbb{Z}$  by \cite{EHW}.

		Thus,  we have
\begin{align*}
z_k(\lambda_0)\leq z &\iff  \exists~ \alpha \in A_{k}: (\lambda_0+ z\zeta+\rho, \alpha^{\vee})>0\\
&\iff A_{k} \nsubseteq \{\alpha \in  
\Delta(\mathfrak{p}^+)\mid (\lambda+\rho,\alpha^{\vee})\leq 0\}\\
&\iff m(\lambda)<k+1=|A_k|\\
&\iff  m(\lambda)\leq k.
\end{align*}

		
		The formula then follows from our main theorem (in type $ADE$) in \cite{BHXZ}. 
		
	(b)	In type $BC$, if $\lambda=\lambda_0+z\zeta$ is a reduction point, then $z\in \frac{1}{2}\mathbb{Z}$ by \cite{EHW}.
		
		When  $\lambda=\lambda_0+z\zeta$ is integral, then $z\in \mathbb{Z}$. Thus,  when $k=2m$ is even, we have
		\begin{align*}
	z_k(\lambda_0)\leq z &\iff  \exists~ \alpha \in A_{k}: (\lambda_0+ z\zeta+\rho, \alpha^{\vee})>0\\
	&\iff A_{k} \nsubseteq \{\alpha \in  
	\Delta(\mathfrak{p}^+)\mid (\lambda+\rho,\alpha^{\vee})\leq 0\}\\
	&\iff m(\lambda)<m+1=|A_k|\\
	&\iff  m(\lambda)\leq m\\
	&\iff  2m(\lambda)\leq 2m=k.
	\end{align*}	
	
		When  $\lambda=\lambda_0+z\zeta$ is half-integral, then $z\in \frac{1}{2}+\mathbb{Z}$. Thus,  when $k=2m+1$ is odd, we have
	\begin{align*}
	z_k(\lambda_0)\leq z &\iff  \exists~ \alpha \in A_{k}: (\lambda_0+ z\zeta+\rho, \alpha^{\vee})>0\\
	&\iff  A_{k} \nsubseteq \{\alpha \in  
	\Delta(\mathfrak{p}^+)\mid (\lambda+\rho,\alpha^{\vee})\leq 0\}\\
	&\iff  m(\lambda)< m+1=|A_k|\\
	&\iff  m(\lambda)\leq m\\
	&\iff   2m(\lambda)+1\leq  2m+1=k.
	\end{align*}
	
In particular, $	z_1(\lambda_0)\leq z \iff  2m(\lambda)+1\leq  2m+1=1$. When $\lambda=\lambda_0+z\zeta$ is half-integral, we also know that 	$L(\lambda)$ is not finite-dimensional. Thus $\AV(L(\lambda)) =\overline{\mathcal{O}_{k(\lambda)}}$ with $k(\lambda)=2m(\lambda)+1\geq 1.$

	The formula then follows from our main theorem (in type $BC$) in \cite{BHXZ}. 		
		
	\end{proof} 
The following result was firstly proved in \cite{BH}. Now we give a new proof.
\begin{cor}
    Suppose $L(\lambda)$ is a unitary highest weight Harish-Chandra module with highest weight $\lambda$. We denote $z=z(\lambda)=(\lambda+\rho,\beta^{\vee})$, then
		\begin{align*}\gk L(\lambda)=
		\begin{cases}
		rz_{r-1}, & \mbox{if  $z<z_{r-1}$}\\
		kz_{k-1}, & \mbox{if  $z=z_{k}=(\rho,\beta^{\vee})-kc, 1\leq k\leq r-1$}\\
		0, & \mbox{if $z=z_{0}=(\rho,\beta^{\vee})$}.
		\end{cases}
		\end{align*}
\end{cor}
\begin{proof}
First we suppose that $L(\lambda)$ is a unitary highest weight Harish-Chandra module and $z=(\lambda+\rho,\beta^{\vee})=z_k$ for some $0\leq k\leq r-1$.
    
From Yamashita \cite{Ya-94} we have $\AV(L(\lambda))=\overline{\mathcal{O}_{k(\lambda)}}$ for some $0\leq k(\lambda)\leq r$. From Theorem \ref{conj}, we will have $z=(\lambda+\rho,\beta^{\vee})=z_{k(\lambda)}$ since $L(\lambda)$ is unitarizable. Thus we must have $z=z_k=z_{k(\lambda)}$, which implies that $k(\lambda)=k$. So we must have $\gk L(\lambda)=\dim \mathcal{O}_{k}=kz_{k-1}$.

Now we suppose that $L(\lambda)$ is a unitary highest weight Harish-Chandra module and $z=(\lambda+\rho,\beta^{\vee})<z_{r-1}$. From Yamashita \cite{Ya-94} we still have $\AV(L(\lambda))=\overline{\mathcal{O}_{k(\lambda)}}$ for some $0\leq k(\lambda)\leq r$. If $k(\lambda)\leq r-1$, by Theorem \ref{conj} we will have $z=(\lambda+\rho,\beta^{\vee})=z_{k(\lambda)}$ since $L(\lambda)$ is unitarizable. From our assumption, we will have $z=z_{k(\lambda)}<z_{r-1}$,
which implies that $$(\rho,\beta^{\vee})-k(\lambda)c<(\rho,\beta^{\vee})-(r-1)c\Rightarrow r-1<k(\lambda)\leq r-1.$$
This is a contradiction! So we must have $k(\lambda)=r$  and $\gk L(\lambda)=\dim \mathcal{O}_{r}=rz_{r-1}$.
\end{proof}

\begin{ex}
Let $\mathfrak{g}_\R=\mathfrak{su}(4,3)$ and let $L(\lambda)$ be a highest weight Harish-Chandra module with highest weight $\lambda=\lambda_0+z\zeta$. Here $\lambda_0=(0,0,0,-20,8,6,6)$, $\zeta=(\frac{3}{7},\frac{3}{7},\frac{3}{7},\frac{3}{7},-\frac{4}{7},-\frac{4}{7},-\frac{4}{7})$ and $\rho=(3,2,1,0,-1,-2,-3)$.
From \cite{EHW}, we know the unitary reduction points correspond to $z=3$ and $4$. 	For  $0\leq k\leq r-1=\min\{p,q\}-1=2$, we know
$$
A_k=\Delta(\p^+)_{k+1}=\{\varepsilon_{p-k+i} -\varepsilon_{p+i+1} = [\, \underbrace{0,\ldots,0}_{p-k+i-1},\underbrace{1,\ldots,1}_{k+1} ,\underbrace{0,\ldots,0}_{q-i-1}\, ] \mid 0\leq i \leq k\}.
$$
So $A_0=\{\varepsilon_{4} -\varepsilon_{5}\}$, $A_1=\{\varepsilon_{3} -\varepsilon_{5},\varepsilon_{4} -\varepsilon_{6}\}$, $A_2=\{\varepsilon_{2} -\varepsilon_{5},\varepsilon_{3} -\varepsilon_{6},\varepsilon_{4} -\varepsilon_{7}\}$.

We write $\lambda_0+ z\zeta+\rho=(3+\frac{3}{7}z,2+\frac{3}{7}z,1+\frac{3}{7}z,-20+\frac{3}{7}z,7-\frac{4}{7}z,4-\frac{4}{7}z,3-\frac{4}{7}z)$.
Thus we have $z_0(\lambda_0)=28$, $z_1(\lambda_0)=7$, and $z_2(\lambda_0)=4$.
 So 	\begin{align}
 \gk L(\lambda)&=
 \begin{cases}
 3z_{2}, & \mbox{if $z<z_{2}(\lambda_0)=4$}\\
 kz_{k-1}, & \mbox{if $z_{k}(\lambda_0)\leq z < z_{k-1}(\lambda_0)$, where $1\leq k\leq 2$}  \\
 0, &  \mbox{if $z_{0}(\lambda_0)=28 \leq z\in \mathbb{Z}$}.\\
 \end{cases}\\
 &=\begin{cases}
 12, & \mbox{if $z<z_{2}(\lambda_0)=4$}\\
 10, & \mbox{if $z_{2}(\lambda_0)=4\leq z < z_{1}(\lambda_0)=7$}  \\
 6, & \mbox{if $z_{1}(\lambda_0)=7\leq z < z_{0}(\lambda_0)=28$}  \\
 0, &  \mbox{if $z_{0}(\lambda_0)=28 \leq z\in \mathbb{Z}$}.\\
 \end{cases}\label{Gk}
 \end{align}
For this given  $\lambda_0$, when $\AV(L(\lambda))=\overline{\O_k}$ with $0\leq k\leq 2$, from  Theorem \ref{conj}, we have  
$$
L(\lambda)
\text{ ~is unitarizable if and only if ~} (\lambda+\rho,\beta^\vee)= z_k=n-1-k=6-k.
$$
So from the above equation (\ref{Gk}), we have $$
L(\lambda)
\text{ ~is unitarizable if and only if ~}(\lambda+\rho,\beta^\vee)=z=z_2=4.
$$
Note that $z=3$ is a unitary reduction point with $\AV(L(\lambda))=\overline{\O_3}$, which is not included in our Theorem \ref{conj}.
\end{ex}

\begin{ex}
	Let $\mathfrak{g}_\R=\mathfrak{sp}(6,\mathbb{R})$ and let  $L(\lambda)$ be a highest weight Harish-Chandra module with  highest weight $\lambda=\lambda_0+z\zeta$. Here $\lambda_0=(-6,-6,-6,-6,-10,-15)$, $\zeta=(1,1,1,1,1,1)$ and $\rho=(6,5,4,3,2,1)$.
	From \cite{EHW}, we know the unitary reduction points correspond to $z=2.5, 3, 3.5$ and $4$. 	
	For  $0\leq k\leq r-1=n-1=5$, we know
	\begin{align*}
	\text{if $k=2m$ is even}, A_k&=\Delta(\mathfrak{p}^+)_{k+1}=\{2\varepsilon_{n-m},\varepsilon_{n-m-i}+\varepsilon_{n-m+i}\mid 1\leq i\leq m\}\\
	&=\{[\, \underbrace{0,\ldots,0}_{n-m-1-i},\underbrace{1,\ldots,1}_{2i},\underbrace{2,\ldots,2}_{m-i},1\, ]\mid 0\leq i \leq m\},
	\end{align*}
	\begin{align*}
	\text{if $k=2m+1$ is odd}, A_k&=\Delta(\mathfrak{p}^+)_{k+1}=\{\varepsilon_{n-m-1-i}+\varepsilon_{n-m+i}\mid 0\leq i\leq m\}\\
	&=\{[\, \underbrace{0,\ldots,0}_{n-m-2-i},\underbrace{1,\ldots,1}_{1+2i},\underbrace{2,\ldots,2}_{m-i},1\, ]\mid 0\leq i \leq m\}.
	\end{align*}
	So $A_0=\{2\varepsilon_{6} \}$, $A_2=\{2\varepsilon_{5},\varepsilon_{4} +\varepsilon_{6}\}$, $A_4=\{2\varepsilon_{4},\varepsilon_{3} +\varepsilon_{5},\varepsilon_{2} +\varepsilon_{6}\}$, $A_1=\{\varepsilon_{5}+\varepsilon_{6} \}$, $A_3=\{\varepsilon_{4}+\varepsilon_{5},\varepsilon_{3} +\varepsilon_{6}\}$,   $A_5=\{\varepsilon_{3}+\varepsilon_{4},\varepsilon_{2} +\varepsilon_{5},\varepsilon_{1} +\varepsilon_{6}\}$.
	
	We write $\lambda_0+ z\zeta+\rho=(z,z-1,z-2,z-3,z-8,z-14)$.
	Thus we have $z_0(\lambda_0)=15$,  $z_2(\lambda_0)=9$, $z_4(\lambda_0)=4$, $z_1(\lambda_0)=11.5$, $z_3(\lambda_0)=6.5$, $z_5(\lambda_0)=3.5$.
	So 	\begin{align}
	\gk L(\lambda)&=
	\begin{cases}
	6z_{5}, & \mbox{if $z<z_{5}(\lambda_0)=3.5$}\\
	kz_{k-1}, & \mbox{if $z_{k}(\lambda_0)\leq z < z_{k-2}(\lambda_0)$, where $2\leq k\leq 5$ and}  \\
	& \mbox{\quad either $z\in \mathbb{Z}$ and  $k$ is even  or $z\in \frac{1}{2}+\mathbb{Z}$ and $k$ is odd}\\
	6, & \mbox{if $z_{1}(\lambda_0)=11.5  \leq z\in \frac{1}{2}+\mathbb{Z}$}\\
	0, &  \mbox{if $z_{0}(\lambda_0)=15  \leq z\in \mathbb{Z}$}.\\
	\end{cases}\\
	&=\begin{cases}
	21, & \mbox{if $z<z_{5}(\lambda_0)=3.5$}\\
	20, &\mbox{if $z_{5}(\lambda_0)=3.5\leq z < z_{3}(\lambda_0)=6.5$, where $z\in \frac{1}{2}+\mathbb{Z}$}  \\
	18, & \mbox{if $z_{4}(\lambda_0)=4\leq z < z_{2}(\lambda_0)=9$, where $z\in \mathbb{Z}$}  \\
	15, & \mbox{if $z_{3}(\lambda_0)=6.5\leq z < z_{1}(\lambda_0)=11.5$, where $z\in \frac{1}{2}+\mathbb{Z}$}  \\
	11, & \mbox{if $z_{2}(\lambda_0)=9\leq z < z_{0}(\lambda_0)=15$, where $z\in \mathbb{Z}$}  \\
	6, & \mbox{if $z_{1}(\lambda_0)=11.5  \leq z\in \frac{1}{2}+\mathbb{Z}$}\\
	0, &  \mbox{if $z_{0}(\lambda_0)=15  \leq z\in \mathbb{Z}$}.\\
	\end{cases}\label{GKc}
	\end{align}
For this given  $\lambda_0$, when $\AV(L(\lambda))=\overline{\O_k}$ with $0\leq k\leq 5$, from Theorem \ref{conj}, we have  
$$L(\lambda)
\text{ ~is unitarizable if and only if ~}(\lambda+\rho,\beta^\vee)=z= z_k=6-\frac{k}{2}.
$$
So from the above equation (\ref{GKc}), we have $$
L(\lambda)
\text{ ~is unitarizable if and only if ~} (\lambda+\rho,\beta^\vee)=z=z_4=4, \text{or~} z=z_5=3.5.
$$
Note that $z=3$ and $z=2.5$  are two  unitary reduction points with $\AV(L(\lambda))=\overline{\O_6}$, which are not included in our Theorem \ref{conj}.

\end{ex}

	\subsection*{Acknowledgments}
	
	Z. Bai was supported  by the National Natural Science Foundation of 
	China (No. 12171344).


\section{Appendix}
The diagrams of $\Delta{(\mathfrak{p}^{+})}$. 

$\mathfrak{g}_{\R}=\mathfrak{su}(p,q)$:
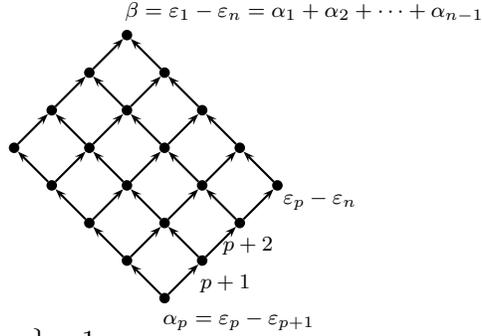
\begin{center}
\hspace{-6pc}\
\begin{pspicture}(-2,0)(2,4)
\psset{linecolor=black}
\cnode*(0,0){.07}{a0}
\cnode*(-.5,.5){.07}{a1}
\cnode*(.5,.5){.07}{a2}
\cnode*(-1,1){.07}{a3}
\cnode*(0,1){.07}{a4}
\cnode*(1,1){.07}{a5}
\cnode*(-1.5,1.5){.07}{a6}
\cnode*(-.5,1.5){.07}{a7}
\cnode*(.5,1.5){.07}{a8}
\cnode*(1.5,1.5){.07}{a9}
\cnode*(-2,2){.07}{a10}
\cnode*(-1,2){.07}{a11}
\cnode*(0,2){.07}{a12}
\cnode*(1,2){.07}{a13}
\cnode*(-1.5,2.5){.07}{a14}
\cnode*(-.5,2.5){.07}{a15}
\cnode*(.5,2.5){.07}{a16}
\cnode*(-1,3){.07}{a17}
\cnode*(0,3){.07}{a18}
\cnode*(-.5,3.5){.07}{a19}
\psset{linecolor=black}
\ncline{->}{a0}{a1}
\ncline{->}{a0}{a2}
\ncline{->}{a1}{a3}
\ncline{->}{a1}{a4}
\ncline{->}{a2}{a4}
\ncline{->}{a2}{a5}
\ncline{->}{a3}{a6}
\ncline{->}{a3}{a7}
\ncline{->}{a4}{a7}
\ncline{->}{a4}{a8}
\ncline{->}{a5}{a8}
\ncline{->}{a5}{a9}
\ncline{->}{a6}{a10}
\ncline{->}{a6}{a11}
\ncline{->}{a7}{a11}
\ncline{->}{a7}{a12}
\ncline{->}{a8}{a12}
\ncline{->}{a8}{a13}
\ncline{->}{a9}{a13}
\ncline{->}{a10}{a14}
\ncline{->}{a11}{a14}
\ncline{->}{a11}{a15}
\ncline{->}{a12}{a15}
\ncline{->}{a12}{a16}
\ncline{->}{a13}{a16}
\ncline{->}{a14}{a17}
\ncline{->}{a15}{a17}
\ncline{->}{a15}{a18}
\ncline{->}{a16}{a18}
\ncline{->}{a17}{a19}
\ncline{->}{a18}{a19}
\uput[r](-.2,-.3){\scriptsize{$\alpha_p=\varepsilon_p-\varepsilon_{p+1}$}}
\uput[r](.3,.2){\scriptsize{$p+1$}}
\uput[r](.6,.7){\scriptsize{$p+2$}}
\uput[r](1.4,1.3){{\scriptsize{$\varepsilon_p-\varepsilon_n$}}}
\uput[r](-.7,3.8){{\scriptsize{$\beta=\varepsilon_1-\varepsilon_n=\alpha_{1}+\alpha_{2}+\dots+\alpha_{n-1}$}}}
\end{pspicture}
\end{center}

	For  $0\leq k\leq r-1=\min\{p,q\}-1$,
	$$
	A_k=\Delta(\p^+)_{k+1}=\{\varepsilon_{p-k+i} -\varepsilon_{p+i+1} = [\, \underbrace{0,\ldots,0}_{p-k+i-1},\underbrace{1,\ldots,1}_{k+1} ,\underbrace{0,\ldots,0}_{q-i-1}\, ] \mid 0\leq i \leq k\}.
	$$

$\mathfrak{g}_{\R}=\mathfrak{so}^*(2n)$:

\hspace{-4pc}\ 	
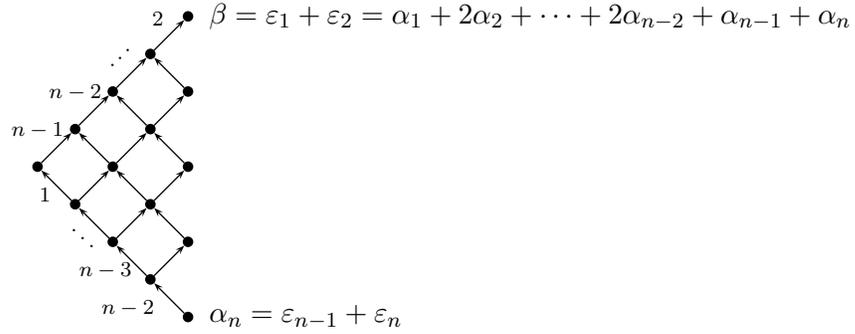
\begin{pspicture}(1,-1)(10,5)
	\psset{linewidth=.5pt,labelsep=8pt,nodesep=0pt}

	\cnode*(8,0){.07}{a0}\uput[r](8,0){$\alpha_{n}=\varepsilon_{n-1}+\varepsilon_{n}$}
	\cnode*(7.5,.5){.07}{a1}
	\cnode*(7,1){.07}{a2}
	\cnode*(8,1){.07}{c2}
	\cnode*(6.5,1.5){.07}{a3}
	\cnode*(7.5,1.5){.07}{a4}
	\cnode*(6,2){.07}{a5}
	\cnode*(7,2){.07}{a6}
	\cnode*(8,2){.07}{c6}
	\cnode*(6.5,2.5){.07}{a7}
	\cnode*(7.5,2.5){.07}{a8}
	\cnode*(7,3){.07}{a9}
	\cnode*(8,3){.07}{a10}
	\cnode*(7.5,3.5){.07}{a12}
	\cnode*(8,4){.07}{c13}\uput[r](8,4){$\beta=\varepsilon_1+\varepsilon_2=\alpha_1+2\alpha_2+\dots+2\alpha_{n-2}+\alpha_{n-1}+\alpha_n$}
	\ncline{->}{a0}{a1}
	\ncline{->}{a1}{a2}
	\ncline{->}{a1}{c2}
	\ncline{->}{a2}{a3}
	\ncline{->}{a2}{a4}
	\ncline{->}{c2}{a4}
	\ncline{->}{a3}{a5}
	\ncline{->}{a3}{a6}
	\ncline{->}{a4}{a6}
	\ncline{->}{a4}{c6}
	\ncline{->}{a5}{a7}
	\ncline{->}{a6}{a7}
	\ncline{->}{a6}{a8}
	\ncline{->}{a7}{a9}
	\ncline{->}{c6}{a8}
	\ncline{->}{a8}{a9}
	\ncline{->}{a8}{a10}
	\ncline{->}{a9}{a12}
	\ncline{->}{a10}{a12}
	\ncline{->}{a12}{c13}
	\uput[d](7.2,.5){\scriptsize{$n-2$}}
	\uput[d](6.9,1){\scriptsize{$n-3$}}
	\uput[d](6.6,1.7){\scriptsize{$\ddots$}}
	\uput[d](6.1,2){\scriptsize{1}}
	\uput[u](6,2.1){\scriptsize{$n-1$}}
	\uput[u](6.5,2.6){\scriptsize{$n-2$}}
	\uput[u](7.1,3.1){\scriptsize{$\cdot$}}
	\uput[u](7.2,3.2){\scriptsize{$\cdot$}}
	\uput[u](7.0,3.0){\scriptsize{$\cdot$}}
	\uput[u](7.6,3.6){\scriptsize{2}}
	\end{pspicture}

 For  $1\leq k\leq r-1=[\frac{n}{2}]-1$,
	\begin{align*}
	A_k&=\Delta(\mathfrak{p}^+)_{2k+1}=\{[\, \underbrace{0,\ldots,0}_{n-2k-2},\underbrace{1,\ldots,1}_{2k} ,0,1\, ],\\
	&~~~~~~~~~\quad\quad\quad\quad\quad\quad\quad[\, \underbrace{0,\ldots,0}_{n-2k-1+i},\underbrace{1,\ldots,1}_{2k-2i-1},\underbrace{2,\ldots,2}_{i} , 1,1\, ] \mid 0\leq i \leq k-1\},\\
	A_0&= \Delta(\mathfrak{p}^+)_{1}=\{[0,0,\dots,0,1]\}.
	\end{align*}

 \vspace{2cm}

\newpage
$\mathfrak{g}_{\R}=\mathfrak{so}(2,2n-2)$:

\hspace{11pc}\ 
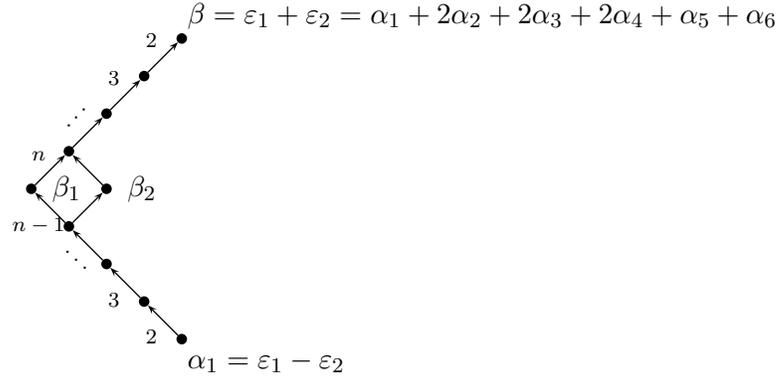
\begin{pspicture}(1,-1)(10,5)
\psset{linewidth=.5pt,labelsep=8pt,nodesep=0pt}
\cnode*(2,0){.07}{a0}
\cnode*(1.5,.5){.07}{a1}
\cnode*(1,1){.07}{a2}
\cnode*(.5,1.5){.07}{a3}
\cnode*(0,2){.07}{a5}
\cnode*(1,2){.07}{a6}
\uput[r](0,2){$\beta_{1}$ }
\uput[r](1,2){$\beta_{2}$ }
\cnode*(.5,2.5){.07}{a7}
\cnode*(1,3){.07}{a9}
\cnode*(1.5,3.5){.07}{a12}
\cnode*(2,4){.07}{b13}
\ncline{->}{a0}{a1}
\ncline{->}{a1}{a2}
\ncline{->}{a2}{a3}
\ncline{->}{a3}{a5}
\ncline{->}{a3}{a6}
\ncline{->}{a5}{a7}
\ncline{->}{a6}{a7}
\ncline{->}{a7}{a9}
\ncline{->}{a9}{a12}
\ncline{->}{a12}{b13}
\uput[d](1.6,.4){\scriptsize{2}}
\uput[d](1.1,.9){\scriptsize{3}}
\uput[d](0.6,1.7){\scriptsize{$\ddots$}}
\uput[d](0.1,1.9){\scriptsize{$n-1$}}
\uput[u](0.1,2.1){\scriptsize{$n$}}
\uput[u](0.5,2.5){\scriptsize{$\cdot$}}
\uput[u](0.6,2.6){\scriptsize{$\cdot$}}
\uput[u](0.7,2.7){\scriptsize{$\cdot$}}
\uput[u](1.1,3.1){\scriptsize{3}}
\uput[u](1.6,3.6){\scriptsize{2}}
\uput[r](1.8,-.3){$\alpha_{1}=\varepsilon_1-\varepsilon_2$}
\uput[r](1.8,4.3){$\beta =\varepsilon_1+\varepsilon_2= \alpha_{1}+2\alpha_{2}+2\alpha_{3}+2\alpha_{4}+\alpha_{5}+\alpha_{6}$}
\end{pspicture}\\[.5 pc]
	\begin{align*}
	A_0&=\Delta(\mathfrak{p}^+)_{1}=\{\varepsilon_{1}-\varepsilon_{2}=[1,0,...,0]\},\\
	A_1&= \Delta(\mathfrak{p}^+)_{n-1}=\{\varepsilon_{1}-\varepsilon_{n}= [1,\dots,1,0], \varepsilon_{1}+\varepsilon_{n}= [1,\dots,1,0,1]\}.
	\end{align*}

$\mathfrak{g}_{\R}=\mathfrak{e}_{6(-14)}$:

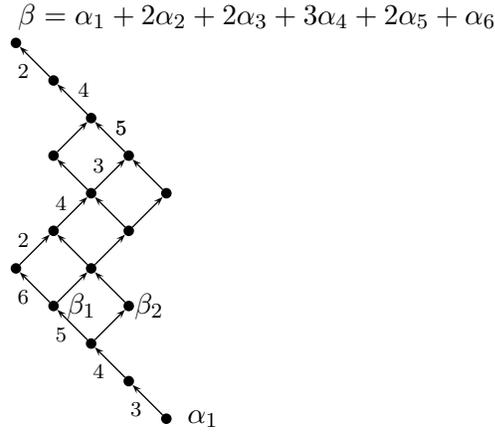
\begin{center}
	\begin{pspicture}(0,0)(10,5.5)
	\psset{linewidth=.5pt,labelsep=8pt,nodesep=0pt}
	\cnode*(2,0){.07}{a0}
	\cnode*(1.5,.5){.07}{a1}
	\cnode*(1,1){.07}{a2}
	\cnode*(.5,1.5){.07}{a3}
	\cnode*(1.5,1.5){.07}{a4}
	\cnode*(0,2){.07}{a5}
	\cnode*(1,2){.07}{a6}
	\cnode*(.5,2.5){.07}{a7}
	\cnode*(1.5,2.5){.07}{a8}
	\cnode*(1,3){.07}{a9}
	\cnode*(2,3){.07}{a10}
	\cnode*(.5,3.5){.07}{a11}
	\cnode*(1.5,3.5){.07}{a12}
	\cnode*(1,4){.07}{a13}
	\cnode*(.5,4.5){.07}{a14}
	\cnode*(0,5){.07}{a15}
	\ncline{->}{a0}{a1}
	\ncline{->}{a1}{a2}
	\ncline{->}{a2}{a3}
	\ncline{->}{a2}{a4}
	\ncline{->}{a3}{a5}
	\ncline{->}{a3}{a6}
	\ncline{->}{a4}{a6}
	\ncline{->}{a5}{a7}
	\ncline{->}{a6}{a7}
	\ncline{->}{a6}{a8}
	\ncline{->}{a7}{a9}
	\ncline{->}{a8}{a9}
	\ncline{->}{a8}{a10}
	\ncline{->}{a9}{a11}
	\ncline{->}{a9}{a12}
	\ncline{->}{a10}{a12}
	\ncline{->}{a11}{a13}
	\ncline{->}{a12}{a13}
	\ncline{->}{a13}{a14}
	\ncline{->}{a14}{a15}
	\uput[r](2,0){$\alpha_{1}$ }
	\uput[r](-.25,5.35){$\beta=  \alpha_{1}+2\alpha_{2}+2\alpha_{3}+3\alpha_{4}+2\alpha_{5}+\alpha_{6}$}
	\uput[r](0.4,1.5){$\beta_{1}$ }
	\uput[r](1.3,1.5){$\beta_{2}$ }
	\uput[d](1.6,.5){\scriptsize{3}}
	\uput[d](1.1,1){\scriptsize{4}}
	\uput[d](0.6,1.5){\scriptsize{5}}
	\uput[d](0.1,2){\scriptsize{6}}
	\uput[u](0.1,2){\scriptsize{2}}
	\uput[u](0.6,2.5){\scriptsize{4}}
	\uput[u](1.1,3){\scriptsize{3}}
	\uput[u](1.4,3.5){\scriptsize{5}}
	\uput[u](1.4,3.5){\scriptsize{5}}
	\uput[u](0.9,4.0){\scriptsize{4}}
	\uput[d](0.1,5.0){\scriptsize{2}}
	\end{pspicture}
\end{center}

\begin{align*}
	A_0&= \Delta(\mathfrak{p}^+)_{1}=\{[1,0,0,0,0,0]\},\\ 
	A_1&= \Delta(\mathfrak{p}^+)_{4}=\{ [1, 0, 1, 1, 1, 0],[1, 1, 1, 1, 0, 0]\}.
	\end{align*}

\newpage
$\mathfrak{g}_{\R}=\mathfrak{e}_{7(-25)}$:

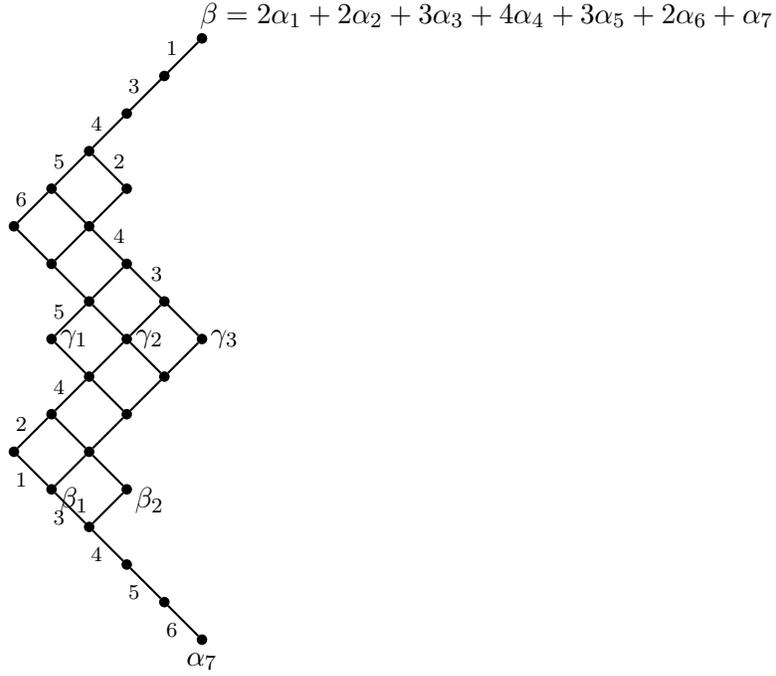
\begin{center}
	\begin{pspicture}(0,-1)(10, 8.5)
	\cnode*(2.5,-.5){.07}{b0}
	\cnode*(2,0){.07}{a0}
	\cnode*(1.5,.5){.07}{a1}
	\cnode*(1,1){.07}{a2}
	\cnode*(.5,1.5){.07}{a3}
	\cnode*(1.5,1.5){.07}{a4}
	\cnode*(0,2){.07}{a5}
	\cnode*(1,2){.07}{a6}
	\cnode*(.5,2.5){.07}{a7}
	\cnode*(1.5,2.5){.07}{a8}
	\cnode*(1,3){.07}{a9}
	\cnode*(2,3){.07}{a10}
	\cnode*(.5,3.5){.07}{a11}
	\cnode*(1.5,3.5){.07}{a12}
	\cnode*(1,4){.07}{a13}
	\cnode*(.5,4.5){.07}{a14}
	\cnode*(0,5){.07}{a15}
	\cnode*(2.5,3.5){.07}{b10}
	\cnode*(2,4){.07}{b12}
	\cnode*(1.5,4.5){.07}{b13}
	\cnode*(1,5){.07}{b14}
	\cnode*(.5,5.5){.07}{b15}
	\cnode*(1.5,5.5){.07}{c14}
	\cnode*(1,6){.07}{c15}
	\cnode*(1.5,6.5){.07}{c16}
	\cnode*(2,7){.07}{c17}
	\cnode*(2.5,7.5){.07}{c18}
	\ncline{a0}{b0}
	\ncline{a0}{a1}
	\ncline{a1}{a2}
	\ncline{a2}{a3}
	\ncline{a2}{a4}
	\ncline{a3}{a5}
	\ncline{a3}{a6}
	\ncline{a4}{a6}
	\ncline{a5}{a7}
	\ncline{a6}{a7}
	\ncline{a6}{a8}
	\ncline{a7}{a9}
	\ncline{a8}{a9}
	\ncline{a8}{a10}
	\ncline{a9}{a11}
	\ncline{a9}{a12}
	\ncline{a10}{a12}
	\ncline{a11}{a13}
	\ncline{a12}{a13}
	\ncline{a13}{a14}	
	\ncline{a14}{a15}
	\ncline{b10}{b15}
	\ncline{a10}{b10}
	\ncline{a12}{b12}
	\ncline{a13}{b13}
	\ncline{a14}{b14}
	\ncline{a15}{b15}
	\ncline{c14}{b14}
	\ncline{c15}{b15}
	\ncline{c14}{c15}
	\ncline{c15}{c18}
	\uput[d](2.1,-.1){\scriptsize{6}}
	\uput[d](1.6,.4){\scriptsize{5}}
	\uput[d](1.1,.9){\scriptsize{4}}
	\uput[d](0.6,1.4){\scriptsize{3}}
	\uput[d](0.1,1.9){\scriptsize{1}}
	\uput[u](0.1,2.1){\scriptsize{2}}
	\uput[u](0.6,2.6){\scriptsize{4}}
	\uput[u](0.6,3.6){\scriptsize{5}}
	\uput[u](1.9,4.1){\scriptsize{3}}
	\uput[u](1.4,4.6){\scriptsize{4}}
	\uput[u](1.4,5.6){\scriptsize{2}}
	\uput[u](0.1,5.1){\scriptsize{6}}
	\uput[u](0.6,5.6){\scriptsize{5}}
	\uput[u](1.1,6.1){\scriptsize{4}}
	\uput[u](1.6,6.6){\scriptsize{3}}
	\uput[u](2.1,7.1){\scriptsize{1}}
	\uput[d](2.5,-.5){$\alpha_{7}$}
	\uput[d](0.8,1.7){$\beta_{1}$}
	\uput[d](1.8,1.7){$\beta_{2}$}
	\uput[d](0.8,3.8){$\gamma_{1}$}
	\uput[d](1.8,3.8){$\gamma_{2}$}
	\uput[d](2.8,3.8){$\gamma_{3}$}
	\uput[r](2.3,7.8){$\beta= 2\alpha_{1}+2\alpha_{2}+3\alpha_{3}+4\alpha_{4}+3\alpha_{5}+2\alpha_{6}+\alpha_7$}
	\end{pspicture}
\end{center}


\begin{align*}
	A_0&= \Delta(\mathfrak{p}^+)_{1}=\{[0,0,0,0,0,0,1]\}, \\
	A_1&= \Delta(\mathfrak{p}^+)_{5}=\{[0, 0, 1, 1, 1, 1, 1], [0, 1, 0, 1, 1, 1, 1]\},\\
	A_2&=\Delta(\mathfrak{p}^+)_{9}=\{[1, 1, 2, 2, 1, 1, 1],[1, 1, 1, 2, 2, 1, 1],[0, 1, 1, 2, 2, 2, 1]\}.
	\end{align*}

$\mathfrak{g}_{\R}=\mathfrak{so}(2,2n-1)$:
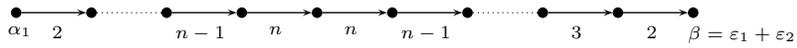
\vspace{1cm}

 \begin{pspicture}(1,-1)(-2,0)
\psset{linewidth=.5pt,labelsep=8pt,nodesep=0pt}
\cnode*(0,0){.07}{d1}
\uput[d](0,0.1){\tiny{ $\alpha_1$}}
\cnode*(1,0){.07}{d2}
\uput[d](0.5,0.1){\tiny{ $2$}}
\cnode*(2,0){.07}{d3}
\uput[d](2.4,0.1){\tiny{ $n-1$}}
\cnode*(3,0){.07}{d4}
\uput[d](3.4,0.1){\tiny{ $n$}}
\uput[d](4.4,0.1){\tiny{ $n$}}
\uput[d](5.4,0.1){\tiny{ $n-1$}}
\cnode*(4,0){.07}{d5}
\cnode*(5,0){.07}{d6}
\cnode*(6,0){.07}{d7}
\uput[d](7.4,0.1){\tiny{ $3$}}
\cnode*(7,0){.07}{d8}
\cnode*(8,0){.07}{d9}
\cnode*(9,0){.07}{d10}
\uput[d](8.4,0.1){\tiny{ $2$}}
\uput[d](9.6,0.1){\tiny{ $\beta=\varepsilon_1+\varepsilon_2$}}
\ncline{->}{d1}{d2}
\ncline[linestyle=dotted,dotsep=1.3pt]{d2}{d3}
\ncline{->}{d3}{d4}
\ncline{->}{d4}{d5}
\ncline{->}{d5}{d6}
\ncline{->}{d6}{d7}
\ncline[linestyle=dotted,dotsep=1.3pt]{d7}{d8}
\ncline{->}{d8}{d9}
\ncline{->}{d9}{d10}
\end{pspicture}\\[.5 pc]
\begin{align*}
	A_0&=\Delta(\mathfrak{p}^+)_{1}=\{\varepsilon_{1}-\varepsilon_{2}=[1,0,...,0]\},\\
	A_1&= \Delta(\mathfrak{p}^+)_{n}=\{\varepsilon_{1}= [1,1,\dots,1]\}.
	\end{align*}

\newpage
$\mathfrak{g}_{\R}=\mathfrak{sp}(n,\mathbb{R})$:
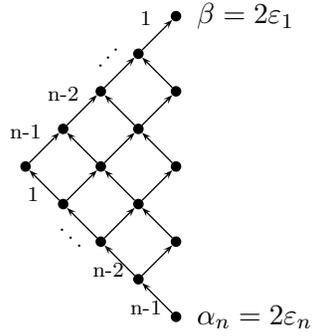
\begin{center}
\begin{pspicture}(1,-1)(10,5)
\psset{linewidth=.5pt,labelsep=8pt,nodesep=0pt}
\cnode*(3,0){.07}{a0}\uput[r](3,0){$\alpha_{n}=2\varepsilon_n$}
\cnode*(2.5,.5){.07}{a1}
\cnode*(2,1){.07}{a2}
\cnode*(3,1){.07}{c2}
\cnode*(1.5,1.5){.07}{a3}
\cnode*(2.5,1.5){.07}{a4}
\cnode*(1,2){.07}{a5}
\cnode*(2,2){.07}{a6}
\cnode*(3,2){.07}{c6}
\cnode*(1.5,2.5){.07}{a7}
\cnode*(2.5,2.5){.07}{a8}
\cnode*(2,3){.07}{a9}
\cnode*(3,3){.07}{a10}
\cnode*(2.5,3.5){.07}{a12}
\cnode*(3,4){.07}{c13}\uput[r](3,4){$\beta=2\varepsilon_1$}
\ncline{->}{a0}{a1}
\ncline{->}{a1}{a2}
\ncline{->}{a1}{c2}
\ncline{->}{a2}{a3}
\ncline{->}{a2}{a4}
\ncline{->}{c2}{a4}
\ncline{->}{a3}{a5}
\ncline{->}{a3}{a6}
\ncline{->}{a4}{a6}
\ncline{->}{a4}{c6}
\ncline{->}{a5}{a7}
\ncline{->}{a6}{a7}
\ncline{->}{a6}{a8}
\ncline{->}{a7}{a9}
\ncline{->}{c6}{a8}
\ncline{->}{a8}{a9}
\ncline{->}{a8}{a10}
\ncline{->}{a9}{a12}
\ncline{->}{a10}{a12}
\ncline{->}{a12}{c13}
\uput[d](2.6,.5){\scriptsize{n-1}}
\uput[d](2.1,1){\scriptsize{n-2}}
\uput[d](1.6,1.7){\scriptsize{$\ddots$}}
\uput[d](1.1,2){\scriptsize{1}}
\uput[u](1,2.1){\scriptsize{n-1}}
\uput[u](1.5,2.6){\scriptsize{n-2}}
\uput[u](2.1,3.1){\scriptsize{$\cdot$}}
\uput[u](2.2,3.2){\scriptsize{$\cdot$}}
\uput[u](2.0,3.0){\scriptsize{$\cdot$}}
\uput[u](2.6,3.6){\scriptsize{1}}

\end{pspicture}
\end{center}

For  $0\leq k\leq r-1=n-1$, 
	\begin{align*}
\text{if $k=2m$ is even}, A_k&=\Delta(\mathfrak{p}^+)_{k+1}=\{2\varepsilon_{n-m},\varepsilon_{n-m-i}+\varepsilon_{n-m+i}\mid 1\leq i\leq m\}\\
&=\{[\, \underbrace{0,\ldots,0}_{n-m-1-i},\underbrace{1,\ldots,1}_{2i},\underbrace{2,\ldots,2}_{m-i},1\, ]\mid 0\leq i \leq m\},
\end{align*}
\begin{align*}
\text{if $k=2m+1$ is odd}, A_k&=\Delta(\mathfrak{p}^+)_{k+1}=\{\varepsilon_{n-m-1-i}+\varepsilon_{n-m+i}\mid 0\leq i\leq m\}\\
&=\{[\, \underbrace{0,\ldots,0}_{n-m-2-i},\underbrace{1,\ldots,1}_{1+2i},\underbrace{2,\ldots,2}_{m-i},1\, ]\mid 0\leq i \leq m\}.
\end{align*}

\bibliography{GK-BH}

\begin{bibdiv}
\begin{biblist}

\bib{BH}{article}{
      author={Bai, Z.},
      author={Hunziker, M.},
       title={The {G}elfand-{K}irillov dimension of a unitary highest weight
  module},
        date={2015},
     journal={Sci. China Math.},
      volume={58},
       pages={2489\ndash 2498},
}

\bib{BHXZ}{article}{
      author={Bai, Z.},
      author={Hunziker, M.},
      author={Xie, X.},
      author={Zierau, R.},
       title={On the associated variety of a highest weight {H}arish-{C}handra
  module},
     journal={arXiv: 2402.08886},
}

\bib{DES}{article}{
      author={Davidson, M.~G.},
      author={Enright, T.~J.},
      author={Stanke, R.~J.},
       title={Differential operators and highest weight representations},
        date={1991},
     journal={Mem. Amer. Math. Soc.},
      volume={94},
      number={455},
       pages={iv+102},
}

\bib{EHW}{incollection}{
      author={Enright, T.~J.},
      author={Howe, R.},
      author={Wallach, N.},
       title={A classification of unitary highest weight modules},
        date={1983},
   booktitle={Representation theory of reductive groups ({P}ark {C}ity, {U}tah,
  1982)},
      series={Progr. Math.},
      volume={40},
   publisher={Birkh\"auser Boston, Boston, MA},
       pages={97\ndash 143},
}

\bib{EW}{article}{
      author={Enright, T.~J.},
      author={Willenbring, J.~F.},
       title={Hilbert series, {H}owe duality and branching for classical
  groups},
        date={2004},
        ISSN={0003-486X},
     journal={Ann. of Math. (2)},
      volume={159},
      number={1},
       pages={337\ndash 375},
}

\bib{Ja1}{article}{
      author={Jakobsen, H.~P.},
       title={Hermitian symmetric spaces and their unitary highest weight
  modules},
        date={1983},
        ISSN={0022-1236},
     journal={J. Funct. Anal.},
      volume={52},
      number={3},
       pages={385\ndash 412},
}

\bib{KV}{article}{
      author={Kashiwara, M.},
      author={Vergne, M.},
       title={On the {S}egal-{S}hale-{W}eil representations and harmonic
  polynomials},
        date={1978},
        ISSN={0020-9910},
     journal={Invent. Math.},
      volume={44},
      number={1},
       pages={1\ndash 47},
}

\bib{NOT}{incollection}{
      author={Nishiyama, K.},
      author={Ochiai, H.},
      author={Taniguchi, K.},
       title={Bernstein degree and associated cycles of {H}arish-{C}handra
  modules---{H}ermitian symmetric case},
        date={2001},
       pages={13\ndash 80},
        note={Nilpotent orbits, associated cycles and Whittaker models for
  highest weight representations},
}

\bib{Vo78}{article}{
      author={Vogan, D.~A., Jr.},
       title={Gelfand-{K}irillov dimension for {H}arish-{C}handra modules},
        date={1978},
        ISSN={0020-9910},
     journal={Invent. Math.},
      volume={48},
      number={1},
       pages={75\ndash 98},
}

\bib{Vo91}{incollection}{
      author={Vogan, D.~A., Jr.},
       title={Associated varieties and unipotent representations},
        date={1991},
   booktitle={Harmonic analysis on reductive groups ({B}runswick, {ME}, 1989)},
      series={Progr. Math.},
      volume={101},
   publisher={Birkh\"auser Boston, Boston, MA},
       pages={315\ndash 388},
}

\bib{Hir01}{incollection}{
      author={Yamashita, H.},
       title={Cayley transform and generalized {W}hittaker models for
  irreducible highest weight modules},
        date={2001},
       pages={81\ndash 137},
        note={Nilpotent orbits, associated cycles and Whittaker models for
  highest weight representations},
}

\bib{Ya-94}{article}{
      author={Yamashita, H.},
       title={Criteria for the finiteness of restriction of {$U({\mathfrak
  g})$}-modules to subalgebras and applications to {H}arish-{C}handra modules.
  {A} study in relation to the associated varieties},
        date={1994},
        ISSN={0022-1236},
     journal={J. Funct. Anal.},
      volume={121},
      number={2},
       pages={296\ndash 329},
         url={http://dx.doi.org/10.1006/jfan.1994.1051},
      review={\MR{1272130}},
}

\end{biblist}
\end{bibdiv}
\end{document}